\documentclass[12pt,a4paper,reqno]{amsart}

\usepackage{xcolor}
\usepackage[all,2cell,arrow,color]{xy}

\usepackage{amssymb}
\usepackage{mathtools}
\usepackage{tensor}
\numberwithin{equation}{section}
\usepackage{graphicx} 
\usepackage{rotating}
\usepackage{pdflscape}
\definecolor{darkgreen}{rgb}{0,0.45,0} 
\usepackage[pagebackref,colorlinks,citecolor=darkgreen,linkcolor=darkgreen]{hyperref}
\usepackage{enumerate,xspace}

\UseAllTwocells
\CompileMatrices

 \definecolor{lightgrey}{rgb}{0.666666,0.666666,0.666666}

\renewcommand{\epsilon}{\varepsilon}
\renewcommand{\phi}{\varphi}

\newcommand{\cb}{\ensuremath{\mathcal B}\xspace}
\newcommand{\cc}{\ensuremath{\mathcal C}\xspace}
\newcommand{\cd}{\ensuremath{\mathcal D}\xspace}

\newcommand{\cm}{\ensuremath{\mathcal M}\xspace}

\newcommand{\ct}{\ensuremath{\mathcal T}\xspace}

\newcommand{\bz}{\ensuremath{\mathbb Z}\xspace}

\newcommand{\Mod}{\ensuremath{\mathsf{Mod}}\xspace }
\newcommand{\Span}{\ensuremath{\mathsf{Span}}\xspace}
\newcommand{\Cat}{\ensuremath{\mathsf{Cat}}\xspace}
\newcommand{\Set}{\ensuremath{\mathsf{Set}}\xspace}
\newcommand{\Prof}{\ensuremath{\mathsf{Prof}}\xspace}
\newcommand{\Vect}{\ensuremath{\mathsf{vec}}\xspace}

\newcommand{\co}{\ensuremath{^{\textnormal{co}}}}
\newcommand{\op}{\ensuremath{^{\textnormal{op}}}}
\newcommand{\rev}{\ensuremath{^{\textnormal{rev}}}}
\newcommand{\coop}{\ensuremath{^{\textnormal{co,op}}}}
\newcommand{\oprev}{\ensuremath{^{\textnormal{op,rev}}}}

\newcommand{\du} {\ensuremath{_{\textnormal{d}}}}

\newcommand{\oprevdl}{\ensuremath{^{\textnormal{op\, rev}}_{\ \ \textnormal{d}}}}

\DeclareMathOperator{\End}{End}

\newcommand{\dual}[1]{{\overline{#1}}}
\newcommand{\ant}{\sigma}

\newcommand{\ld}[1]{\tensor[^\bot]{#1}{}}

\newcommand{\ox}{\otimes}
\newcommand{\x}{\times}

\renewcommand{\phi}{\varphi}

\def\1c#1{\stackrel{#1}{\to}}

  \newtheorem{proposition}{Proposition}[section]
  \newtheorem{lemma}[proposition]{Lemma}

  \newtheorem{theorem}[proposition]{Theorem}

  \theoremstyle{definition}
  
  \newtheorem{example}[proposition]{Example}
  
  \newtheorem{claim}[proposition]{Construction}
  
\theoremstyle{remark}
  \newtheorem{remark}[proposition]{Remark}

  \newcounter{c}
  
  \newcommand{\etyk}[1]{\vspace{-7.4mm}$$\begin{equation}\Label{#1}
  \addtocounter{c}{1}}
  \renewcommand{\]}{\ifnum \value{c}=1 $$\else \end{equation}\fi}
\begin{document}

 \title{Hopf comonads on naturally Frobenius map-monoidales}

\author{Gabriella B\"ohm} 
\address{Wigner Research Centre for Physics, H-1525 Budapest 114,
P.O.B.\ 49, Hungary}
\email{bohm.gabriella@wigner.mta.hu}
\author{Stephen Lack}
\address{Department of Mathematics, Macquarie University NSW 2109, Australia}
\email{steve.lack@mq.edu.au}

\date{Nov 2014}

\begin{abstract}
We study monoidal comonads on a naturally Frobenius map-monoidale $M$ in a
monoidal bicategory $\mathcal M$. We regard them as bimonoids in the duoidal
hom-category $\mathcal M(M,M)$, and generalize to that setting various
conditions distinguishing classical Hopf algebras among bialgebras; in
particular, we define a notion of antipode in that context. Assuming the
existence of certain conservative functors and the splitting of idempotent
2-cells in $\mathcal M$, we show all these Hopf-like conditions to be
equivalent. Our results imply in particular several equivalent
characterizations of Hopf algebras in braided monoidal categories, of small
groupoids, of Hopf algebroids over commutative base algebras, of weak Hopf
algebras, and of Hopf monads in the sense of Brugui\`eres and Virelizier. 
\end{abstract}
  
\maketitle


\section{Introduction}

Classical {\em bialgebras} (say, over a field $k$) are the same as comonoids 
in the monoidal category of $k$-algebras; that is, in the
monoidal category of monoids in the category \Vect of vector spaces over $k$.
They are also the same as monoids in the monoidal category of 
$k$-coalgebras (that is, of comonoids in \Vect).  A {\em Hopf algebra} is then 
a bialgebra $A$ admitting a further map called the {\em antipode}, 
which is the convolution inverse of the identity map $A\to A$. Since an 
inverse is unique whenever it exists, its existence is a property rather than 
an additional structure. In fact, this property has a number of equivalent 
reformulations; all of them of different conceptual meaning. For instance,  
$A$ is known to be a Hopf algebra if and only if the monad $A\ox-$ on
$\Vect$, defined using the algebra structure, is a left Hopf monad in
the sense of \cite{BLV:HopfMonad}; equivalently, if the monad $-\ox A$ is a
right Hopf monad. This is further equivalent to the comonad $A\ox-$, defined
using the coalgebra structure, being a left Hopf comonad, and also to the
comonad $-\ox A$ being a right Hopf comonad. (In each case, the monad or
comonad is Hopf in the two-sided sense just when the antipode is invertible.)
Then again, $A$ is Hopf if and only if the fundamental theorem of Hopf modules
holds, meaning that the category of Hopf modules over $A$ is equivalent to the 
category of vector spaces. Finally, $A$ is Hopf if and only if $A$ is an 
$A$-Galois extension of the base field, or equivalently an $A$-Galois 
coextension. 

Replacing the category of vector spaces above with any braided monoidal
category, one still can define bialgebras (or bimonoids) as monoids in the
monoidal category of comonoids, equivalently, as comonoids in the monoidal
category of monoids. Still more generally, the monoid and comonoid structures
can be defined with respect to different, but appropriately related, monoidal
structures. Categories with such structure were considered in
\cite{AguiarMahajan-monoidal} under the original name {\em 2-monoidal
category} though since then the term {\em duoidal category} (suggested in
\cite{Street-chaire}) seems to be more widely used. A duoidal category is
equipped with two monoidal structures which are compatible, in the sense that
the functors and natural transformations describing the first monoidal
structure, are monoidal with respect to the second monoidal
structure. Equivalently, the functors and natural transformations describing
the second monoidal structure are opmonoidal with respect to the first
monoidal structure. (For a more restrictive notion, where these monoidal
structures are required to preserve the unit strictly, see
\cite{BFSV:ItMonCat}; when the monoidal functors are strong we recover the  
notion of braided monoidal category \cite{JoyalStreet-braided}.) 
 
The first monoidal structure $\circ$ in a duoidal category lifts to the
category of monoids with respect to the second monoidal structure $\bullet$ and
so one can define a {\em bimonoid} as a comonoid in this monoidal category of
monoids. Symmetrically, the monoidal structure $\bullet$ lifts to the category
of comonoids with respect to the monoidal structure $\circ$ and a monoid in
this monoidal category of comonoids yields an equivalent definition of
bimonoid \cite{AguiarMahajan-monoidal}.

There seems to be no consensus, however, on how to define a {\em Hopf monoid} 
in a duoidal category. There are several approaches in the literature: Street 
in \cite{Street-chaire} investigated the invertibility of a canonical
morphism associated to a bimonoid. In \cite{BohmChenZhang}, the relationship
between  the Hopf property of the induced bimonad, an appropriate Galois 
condition, and validity of the fundamental theorem of Hopf modules on a
bimonoid is analyzed. (For discussion of a similar question
see also \cite{AguiarChase:GenHopfMod}.) None of these, however, involved a
notion of antipode. 

Examples of bimonoids in duoidal categories include bimonoids in braided
monoidal categories \cite{Majid}, small categories
\cite{AguiarMahajan-monoidal}, bialgebroids over commutative base algebras
(such that the source and target maps land in the center)
\cite{AguiarMahajan-monoidal}, weak bialgebras \cite{BohmSpanish-CatWkBialg},
as well as opmonoidal monads (so-called bimonads) and monoidal comonads
(so-called bicomonads) on monoidal categories with left and right duals
\cite{BV}. In these motivating examples the existence of a (suitably defined)
antipode turns out to be equivalent to the aforementioned Hopf-like
properties; and the main aim of this paper is to find a conceptual explanation
of this common feature. With this motivation, we study a particular class of
duoidal categories, large enough to include the key examples, and prove that
for these duoidal categories all the Hopf-like conditions seen in the examples
are equivalent.  

The duoidal categories in question have the following form. Consider a
monoidal bicategory $\mathcal M$. It was observed in \cite{Street-chaire} that
if $M$ is a {\em map-monoidale} (i.e. map-pseudomonoid) in \cm, then the
convolution product yields a second monoidal structure on the monoidal
hom-category $\mathcal M(M,M)$ rendering it a duoidal category. A bimonoid
therein is precisely the same as a monoidal comonad on $M$ with respect to the
convolution product. We make the additional assumption on $M$ that it is {\em
naturally Frobenius} \cite{Nacho-FormalHopfAlgebraI,dualsinvert}; that is, its
monoidale and dual comonoidale structures satisfy the Frobenius compatibility
relations. Then $M$ becomes a self-dual object in $\mathcal M$ and taking
mates under this duality defines an equivalence $\mathcal M(M,M) \to \mathcal
M(M,M)$. (The condition of being naturally Frobenius was shown in
\cite{Nacho-FormalHopfAlgebraI} to be equivalent to a ``theorem of Hopf
modules'', albeit of a different type to that which we consider below.)  

We then define the  {\em antipode} for a bimonoid $a$ in the duoidal
category $\mathcal M(M,M)$ to be a 2-cell from $a$ to its image $a^-$ under 
this equivalence. We explain in Theorem~\ref{thm:antipode} the 
sense in which the antipode is a ``convolution inverse'' of the identity 
2-cell $a\to a$, analogously to the case of classical bialgebras. 
Whenever an antipode exists, it is unique and a morphism of monoids and of 
comonoids (cf. Theorem \ref{thm:antipode_(co)monoid_map}). 

Generalizing the equivalent characterizations of a Hopf algebra over a field, 
for any naturally Frobenius map-monoidale $M$ in a monoidal bicategory
$\mathcal M$, and any monoidal comonad $a$ on $M$, we prove in Theorem
\ref{thm:antipode} and Theorem \ref{thm:Hopf} the equivalence of the following 
properties: 
\begin{itemize}
\item $a$ admits an antipode,
\item $a$ is a Hopf monad in $\mathcal M$ (in the sense of \cite{CLS}),
\item $a$ is a Hopf comonad in $\mathcal M$ (in the dual sense).
\end{itemize} 
Under the further assumptions of the existence of certain conservative
functors to $\mathcal M(M,M)$ (called the {\em well-(co)pointedness} of $M$)
and the splitting of idempotent 2-cells in \cm, we prove in Theorem
\ref{thm:coGalois}, Theorem \ref{thm:Galois}, Theorem \ref{thm:fthm}, and
Theorem \ref{thm:dualfthm} that the above properties are further equivalent to
the following ones: 
\begin{itemize}
\item $a$ is an $a$-Galois extension of the unit $j$ of the convolution
  product (in the sense of invertibility of a canonical morphism), 
\item $a$ is an $a$-Galois coextension of the unit $i$ of the composition (in
  the dual sense), 
 \item the fundamental theorem of Hopf modules holds for $a$; that is, the
  category of $a$-Hopf modules is equivalent to the category of $j$-comodules,
\item the dual fundamental theorem of Hopf modules holds for $a$; that is, the
  category of $a$-Hopf modules is equivalent to the category of $i$-modules.
\end{itemize}
Applying these conditions to a bimonoid in a braided monoidal category
(regarded as a monoidal comonad on a suitable naturally Frobenius
map-monoidale), we re-obtain the equivalent characterizations of a Hopf monoid
in \cite[Theorem 3.6]{Vercruysse}. Applying these conditions to a small
category $a$ (regarded as a monoidal comonad on a suitable naturally Frobenius
map-monoidale), all of them are equivalent to $a$ being a groupoid. Applying
these conditions to a bialgebroid $a$ over a commutative algebra (regarded as
a monoidal comonad on a suitable naturally Frobenius map-monoidale), all of
them are equivalent to $a$ being a Hopf algebroid
\cite{Ravenel,Bohm:HoA}. Applying these conditions to a weak bialgebra $a$
(regarded \cite{BohmSpanish-CatWkBialg} as a monoidal comonad on a suitable
naturally Frobenius map-monoidale), all of them are equivalent to $a$ being a
weak Hopf algebra \cite{BNSz:WHAI}.  Finally, applying them to a monoidal
comonad on a monoidal category with left and right duals, seen as a monoidal
comonad in the monoidal bicategory $\Prof$, we recover the notion of (left)
antipode of \cite{BV}.

\subsection*{Acknowledgement}
We wish to express our thanks to Ross Street and Ignacio L\'opez Franco for
their highly useful comments on this work. 
We gratefully acknowledge the financial support of the Hungarian Scientific
Research Fund OTKA (grant K108384), the Australian Research Council Discovery
Grant (DP130101969), an ARC Future Fellowship (FT110100385), and the Nefim
Fund of Wigner RCP. We are each grateful to the warm hospitality of the
other's institution during research visits in Nov-Dec 2013 (Sydney) and
Sept-Oct 2014 (Budapest).   

\section{Naturally Frobenius map-monoidales}

\subsection{Monoidal bicategories}

We work in a monoidal bicategory \cm \cite{GPS,DS:Hopf-algebroid}. By the
coherence theorem of \cite{GPS}, we may write as if \cm were a Gray-monoid:
this is a 2-category equipped with a strictly associative and unital tensor
product, but which may not be strictly functorial.

We denote tensor products by juxtaposition, and the unit by $I$. We write
$M^n$ for the $n$-fold tensor power of an object $M$. The composite of
morphisms $f\colon M\to N$ and $g\colon N\to P$ will generally be denoted by
$g.f$ while the identity morphism on an object $M$ will be written as $1$ or 
$M$, whichever seems clearer in the particular context. 

A morphism $f\colon M\to N$ in a bicategory is sometimes called a {\em map} if
it has a right adjoint. In this case, we generally write $f^*$ for the right
adjoint, and write $\eta_f\colon 1\to f^*.f$ and $\epsilon_f\colon f.f^*\to 1$
for the unit and counit of the adjunction.

For a bicategory \cm, monoidal or otherwise, we write $\cm\op$ for the
bicategory obtained by formally reversing the 1-cells, and $\cm\co$ for the
bicategory obtained by formally reversing the 2-cells, with $\cm\coop$ given
by reversing both. For a monoidal bicategory \cm, we write $\cm\rev$ for the
monoidal bicategory obtained by formally reversing the tensor product.

If $f\dashv f^*$ in \cm, then $f^*\dashv f$ in both $\cm\op$ and $\cm\co$,
while $f\dashv f^*$ in $\cm\coop$. 

\subsection{Monoidales}

A {\em monoidale} (also known as {\em pseudomonoid}) in the monoidal
bicategory \cm consists of an object $M\in\cm$ equipped with 1-cells 
$$\xymatrix@R1pc {
MM \ar[r]^-{m} & M && I \ar[r]^-{u} & M }$$
and  invertible 2-cells in the following diagrams
$$
\xymatrix{
MMM \ar[r]^{m1}_{~}="1" \ar[d]_-{1m} & MM \ar[d]^-{m} 
&&
M \ar[r]^{u1}_{~}="5" \ar@{}[dr]|{~}="6" \ar@/_1pc/[dr]_-{1}  & 
MM \ar[d]^-{m} & M \ar[l]_{1u}^{~}="3" \ar@{}[dl]|{~}="4" 
\ar@/^1pc/[dl]^-{1} \\
MM \ar[r]_{m}^{~}="2" & M 
\ar@{=>}"1";"2"^{\alpha}
&&
& M 
\ar@{=>}"3";"4"_{\rho}
\ar@{=>}"5";"6"_{\lambda} 
}
$$
satisfying coherence conditions like those in the definition of monoidal
category \cite{CWM}. A monoidale in \Cat is just a monoidal
category.

We shall generally leave the 2-cells un-named, and simply speak of a monoidale
$(M,m,u)$ or even just $M$. 

\subsection{Map-monoidales}\label{sec:map_monoidale}

A {\em map-monoidale} is a monoidale $(M,m,u)$ for which $m$ and $u$ have
right adjoints $m^*$ and $u^*$. In this case, the associativity isomorphism
for $m$ induces a coassociativity isomorphism  for $m^*$, and similarly $u^*$
is a counit; thus a map-monoidale $(M,m,u)$ in \cm can equally be seen as a
map-monoidale $(M,m^*,u^*)$ in $\cm\op$ or $\cm\oprev$. We shall write
$(M,m,u)^*$, or simply $M^*$, for $(M,m,u)$ seen as a map-monoidale in
$\cm\oprev$.

\subsection{Naturally Frobenius map-monoidales}

We can consider further compatibility conditions between the monoidal and
comonoidal structures on a map-monoidale $M$. The {\em mates} of the
associativity isomorphism $\alpha\colon m.m1\cong m.1m$ and of 
its inverse $\alpha^{-1}$ are the 2-cells $\pi$ and $\pi'$  
$$\xymatrix{
& MMM \ar[r]^-{m1}  \ar[dr]^-{1m} & 
MM \ar[dr]^-{m} \ar[rr]^{1}_{~}="1" \ar@{=>}"1";[dr]^-{\eta_m}
\ar@{=>}[d]^\alpha & & 
MM \\
MM \ar[ur]^-{1m^*} \ar[rr]_1^{~}="2"  \ar@{=>}[ur];"2"^-{1\epsilon_m} && 
MM \ar[r]_-{m} & 
M \ar[ur]_-{m^*} 
}$$
$$\xymatrix{
& MMM \ar[r]^-{1m}  \ar[dr]^-{m1} & 
MM \ar[dr]^-{m} \ar[rr]^{1}_{~}="1" \ar@{=>}"1";[dr]^-{\eta_m} 
\ar@{=>}[d]^-{\alpha^{-1}} & & MM \\
MM \ar[ur]^-{m^*1} \ar[rr]_1^{~}="2"  \ar@{=>}[ur];"2"^{\epsilon_m1} && MM 
\ar[r]_-{m} & M \ar[ur]_-{m^*} 
}$$
obtained by pasting $\alpha$ and its inverse with the unit $\eta_m$ and counit
$\epsilon_m$ of the adjunction $m\dashv m^*$. 

When $\pi$ and $\pi'$ are invertible, the map-monoidale is said to be {\em
naturally Frobenius} \cite{dualsinvert}. 
If $(M,m,u)$ is a naturally Frobenius map-monoidale in \cm, then $(M,m,u)^*$
is a naturally Frobenius map-monoidale in $\cm\oprev$.  

\section{Duoidal categories arising from map-monoidales}

\subsection{Comonads in bicategories}\label{sec:comonads}

If \cm is a bicategory and $M$ an object of \cm, then the hom-category
$\cm(M,M)$ has a monoidal structure given by horizontal 
composition in \cm. We write $\circ$ for the reverse of this tensor product,
and $i$ for the unit (given by the identity morphism $M\to M$). Thus $f\circ
g$ denotes the composite 
$$
\xymatrix{
M \ar[r]^{f} & M \ar[r]^{g} & M. }
$$

A comonoid in $\cm(M,M)$ is the same as a comonad in \cm on the object $M$. It
consists of a morphism $a\colon M\to M$, equipped with 2-cells $\delta\colon
a\to a\circ a$ and $\epsilon\colon a\to i$ satisfying the usual
coassociativity and counit conditions.

Comonads in \cm are the same as comonads in $\cm\op$.

\subsection{Monoidal morphisms in monoidal bicategories}
\label{sec:monoidal_morphisms}

Now suppose that \cm is a monoidal bicategory. 

If $(M,m,u)$ and $(N,n,v)$ are monoidales in \cm, a {\em monoidal morphism}
from $(M,m,u)$ to $(N,n,v)$ consists of a morphism $a\colon M\to N$ in \cm
equipped with 2-cells 
$$\xymatrix{
MM \ar[r]^{aa}_{~}="1" \ar[d]_-{m} & NN \ar[d]^-{n} &
I \ar[d]_-{u} \ar@{=}[r]_{~}="3" & I \ar[d]^-{v} \\
M \ar[r]_{a}^{~}="2" & N & M \ar[r]_{a}^{~}="4" & N 
\ar@{=>}"1";"2"^{a_2}
\ar@{=>}"3";"4"^{a_0} 
}$$
satisfying associativity and unit conditions analogous to those for a monoidal
functor \cite{CWM}; indeed a monoidal morphism in \Cat is just a
monoidal functor between the corresponding monoidal categories.

If $(M,m,u)$ and $(N,n,v)$ are map-monoidales, then various further phenomena
arise. Pasting $a_2$ and $a_0$ with the counits $\epsilon_m\colon
m.m^*\to1$ and $\epsilon_u\colon u.u^*\to1$ of the adjunctions $m\dashv m^*$
and $u\dashv u^*$ gives 2-cells 
\begin{equation}\label{eq:mu-eta}
\xymatrix{
MM \ar[r]^{aa}_{~}="1" & NN \ar[d]^-{n} &
I \ar@{=}[r]_{~}="3" & I \ar[d]^-{v} \\
M \ar[u]^-{m^*} \ar[r]_{a}^{~}="2" & N & M \ar[u]^-{u^*} \ar[r]_{a}^{~}="4" & N 
\ar@{=>}"1";"2"^-{\mu}
\ar@{=>}"3";"4"^-{\eta} 
}
\end{equation}
and this sets up a bijection between pairs of 2-cells $a_2$ and $a_0$ and
pairs of 2-cells $\mu$ and $\eta$. The associativity and unit conditions on
$a_2$ and $a_0$ can be expressed in terms of $\mu$ and $\eta$, and the result
can be expressed in a particularly simple way.  

To do this, first observe that  $\cm(M,N)$ has a convolution monoidal
structure, with tensor product $x\bullet y$ of $x$ and $y$  given by the
composite  
$$\xymatrix{
M \ar[r]^-{m^*} & MM \ar[r]^-{xy} & NN \ar[r]^-{n} & N }$$
while the unit $j$ is the composite 
$$\xymatrix{
M \ar[r]^-{u^*} & I \ar[r]^-{v} & N .}$$

A monoid in $\cm(M,N)$ consists of a morphism $a\colon M\to N$ equipped with
2-cells $\mu$ and $\eta$ as in \eqref{eq:mu-eta} satisfying associativity and
unit conditions which say precisely that the corresponding $a_2$ and $a_0$
make $a$ into a monoidal morphism $(a,a_2,a_0)$ from $(M,m,u)$ to $(N,n,v)$.  

The 2-cell $a_2\colon n.aa\to a.m$ is obtained by pasting the unit $\eta_m$
of $m\dashv m^*$ onto the left of $\mu\colon n.aa.m^*\to a$; if instead we
pasted the unit $\eta_n$ of $n\dashv n^*$  onto the right, we would obtain a
2-cell $a^2\colon aa.m^*\to n^*.a$. Similarly, pasting the unit $\eta_v$ of
$v\dashv v^*$ onto $\eta$ gives a 2-cell $a^0\colon u^*\to v^*.a$, and the
associativity and unit conditions for $\mu$ and $\eta$ say precisely that
$a^2$ and $a^0$ make $a$ into a monoidal morphism $(a,a^2,a^0)\colon
(N,n^*,v^*)\to(M,m^*,u^*)$ in $\cm\oprev$. 

\subsection{Monoidal comonads and duoidal categories}

Now specialize to the case of a single map-monoidale 
$(M,m,u)=(N,n,v)$ in a monoidal bicategory $\cm$. Then $\cm(M,M)$
has two monoidal structures, with tensor products $\circ$ and $\bullet$ 
as in Sections \ref{sec:comonads} and \ref{sec:monoidal_morphisms},
respectively, which we call the {\em composition} and {\em
convolution} monoidal structures.  

We know that a comonad on $M$ is the same as a comonoid in $\cm(M,M)$ with
respect to composition, and we know that a monoidal endomorphism of $M$ is 
the same as a monoid in $\cm(M,M)$ with respect to convolution. A monoidal 
comonad on $M$ is an endomorphism $a\colon M\to M$ equipped with both a 
comonad structure and monoidal structure, and with compatibility conditions 
between the two requiring the comultiplication and counit to be monoidal 
 2-cells. How can this compatibility be expressed in terms of
$\cm(M,M)$?  

To do this, we use the notion of `2-monoidal category' introduced in
\cite{AguiarMahajan-monoidal}; following Street, however, we use the name {\em
duoidal category} for such a structure. This involves two monoidal structures
$(\cd,\bullet,j)$ and $(\cd,\circ,i)$ on the same category, along with
morphisms
$$
\xymatrix @R0pc {
(w\circ x)\bullet(y\circ z) \ar[r]^{\xi_{w,x,y,z}} & 
(w\bullet y)\circ(x\bullet z) \\
j \ar[r]^{\xi^0} & j\circ j \\
i\bullet i \ar[r]^{\xi_0} & i \\
j \ar[r]^{\xi^0_0} & i }
$$
(natural in $w,x,y,z$) subject to  the following axioms. The datum
$(\circ,\xi,\xi^0)$ is a monoidal functor with respect to the monoidal product
$\bullet$, and the unit and associativity isomorphisms of the $\circ$-product
are $\bullet$-monoidal natural transformations. Equivalently,
$(\bullet,\xi,\xi_0)$ is an opmonoidal functor with respect to the monoidal
product $\circ$, and the unit and associativity isomorphisms of the
$\bullet$-product are $\circ$-opmonoidal natural transformations. More
succinctly, a duoidal category is a monoidale (or pseudo-monoid)
in the 2-category $\mathsf{OpMon}$ of monoidal categories, opmonoidal
functors, and opmonoidal natural transformations. Examples arise via the
``looping principle'' (see \cite[Appendix C]{AguiarMahajan-monoidal}): as
hom-categories $\cc(X,X)$, for any object $X$ in a category $\cc$ enriched in
$\mathsf{OpMon}$. For more details see \cite{AguiarMahajan-monoidal}.  

A key observation of \cite{AguiarMahajan-monoidal} was that it is possible to 
define bialgebras internal to a duoidal category: these have a coalgebra 
structure with respect to $\circ$, an algebra structure with respect to 
$\bullet$, and compatibility conditions between the two, expressed using the 
various maps $\xi$ listed above.

Now in any monoidal bicategory $\cm$, the full sub-bicategory whose 
objects are the map-monoidales (and hence its opposite bicategory), is in fact 
$\mathsf{OpMon}$-enriched. The monoidal structure $\bullet$ of $\cm(M,N)$, for 
map-monoidales $M$ and $N$, was discussed in Section 
\ref{sec:monoidal_morphisms}. The composition $\cm(N,P) \times \cm(M,N)\to 
\cm(M,P)$ (and hence its opposite $\circ$), as well as the unit $1\to 
\cm(M,M)$ are opmonoidal functors, and the coherence natural isomorphisms are 
opmonoidal, with respect to $\bullet$. Thus by the ``looping principle'', 
$\cm(M,M)$ is duoidal for the two monoidal structures introduced above; 
this is essentially the example discussed in \cite[Section~4.6]{Street-chaire}. 
The map $\xi^0_0\colon j\to i$ is the counit $\epsilon_u$ of the adjunction 
$u\dashv u^*$, while $\xi^0\colon j\to j\circ j$ is the comultiplication of the 
induced comonad, and $\xi_0$ is the counit $\epsilon_m$ of the adjunction 
$m\dashv m^*$. Finally $\xi_{w,x,y,z}$ is formed as in the diagram 
$$
\xymatrix{
M \ar[r]^-{m^*} & M^2 \ar[r]^-{wy} & M^2 \ar[r]^-{m}
\ar@/_2pc/@{=}[rr]^{~}="1" 
\ar@{=>}"1";[r]_-{\eta_m}  & M \ar[r]^-{m^*} & M^2 \ar[r]^-{xz} 
& M^2 \ar[r]^-{m} & M}
$$
in which $\eta_m$  is the unit of the adjunction $m\dashv m^*$.

Now a bialgebra in the duoidal category $\cm(M,M)$ is precisely a monoidal
comonad in \cm on the monoidale $(M,m,u)$, hence it induces a monoidal comonad
on $\cm(M,M)$ with respect to the monoidal structure involving $\bullet$: see
\cite{Street-chaire} once again.  

\begin{example}\label{ex:trivial}
The unit of any monoidal category has a trivial monoid and comonoid
structure. In particular, the unit object $i$ for the $\circ$-monoidal
structure has a trivial comonoid structure with respect to $\circ$; but in a
duoidal category, $i$ is also a monoid for the $\bullet$-monoidal structure
via $\xi_0$ and $\xi^0_0$, and the compatibility conditions hold, 
so that $i$ is in fact a bialgebra. We call it the {\em $\circ$-trivial
bialgebra}.   

Similarly, $j$ is a bialgebra with the $\bullet$-monoid structure being
trivial; we call it the {\em $\bullet$-trivial bialgebra}.  
\end{example}

\begin{remark}\label{rmk:double} 
A {\em double algebra} in the sense of \cite{Szlachanyi-double} involves two 
monoid structures subject to certain equations relating the two structures. 
Similarly a duoidal category involves two monoidal structures with various 
structure relating them. Thus one could ask to what extent the axioms of 
\cite{Szlachanyi-double} hold for duoidal categories. One of these axioms, 
translated into our notation, says that $((a\bullet i)\circ j)\bullet b=
(a\bullet i)\circ b$ for any elements $a$ and $b$ of the double algebra. For 
any two objects $a$ and $b$ of a duoidal category, there is a natural map  
$$
\xymatrix @C=8pt {
((a\bullet i)\circ j)\bullet b = ((a\bullet i)\circ j)\bullet (i\circ b) 
\ar[r]^-{\raisebox{7pt}{$_{\xi}$}} & 
(a\bullet i\bullet i)\circ (j\bullet b)= 
(a\bullet i\bullet i)\circ b \ar[r]^-{\raisebox{7pt}{$_{(1\bullet\xi_0)\circ 1}$}} & 
(a\bullet i)\circ b 
}$$
and so the axiom of \cite{Szlachanyi-double} holds in the ``lax'' sense that
there is a comparison map between the two sides. Furthermore, this comparison
map is invertible if the duoidal category arises from a naturally Frobenius
map-monoidale, and so the axiom holds up to isomorphism in that
case. Similarly for each of the other seven axioms in \cite{Szlachanyi-double}
there is a comparison map in any duoidal category, and this is invertible in
the case arising from a naturally Frobenius map-monoidale.  
\end{remark}

\subsection{Hopf map}\label{sec:HopfMap}

For a monoidal comonad $a$ on a monoidal category $M$, and any objects
$x,y$ of $M$, we can form the
composite  
$$\xymatrix{
a(x)\ox a(y) \ar[r]^-{\delta\ox1} & a(a(x))\ox a(y) 
\ar[r]^-{a_2} &
a(a(x)\ox y) }$$ 
which is sometimes called the Hopf map. 
The analogue \cite{CLS} in our internal setting is the 2-cell
$\hat{\beta}\colon m.aa\to a.m.a1$ given by the pasting composite below.
\begin{equation}\label{eq:beta-hat}
 \hat{\beta} \quad=\quad  
\raisebox{18pt}{$\xymatrix{
MM \ar[r]^-{aa}_{~}="1" \ar@{=}[d] & 
MM \ar[r]^-{m}_{~}="3" & M \\
MM \ar[r]_-{a1}^{~}="2" & MM \ar[u]_-{aa} \ar[r]_-{m}^{~}="4" & M \ar[u]_a 
\ar@{=>}"1";"2"^{\delta a} 
\ar@{=>}"3";"4"^{a_2} }$}
\end{equation}
We call $\hat{\beta}$ the {\em Hopf map} associated to the monoidal comonad 
$a$ on the monoidale $M$. In the terminology of \cite{CLS}, $a$ is a {\em
right Hopf comonad} whenever $\hat \beta$ is invertible. 

On the other hand, as observed above in Section~\ref{sec:monoidal_morphisms}, 
whenever $M$ is a map-monoidale,  we can
also think of a monoidal comonad on $(M,m,u)$ as a monoidal comonad on
$(M,m,u)^*$. In this case, the Hopf map is the 2-cell  $\hat{\zeta}\colon
aa.m^*\to 1a.m^*.a$ given by the composite appearing below; we call it the
{\em co-Hopf map}.
\begin{equation}\label{eq:zeta-hat}
\hat{\zeta}\quad = \quad
\raisebox{18pt}{$\xymatrix{
M \ar[r]^-{m^*}_{~}="1" \ar[d]_-{a} & MM \ar[d]^-{aa} \ar[r]^-{aa}_{~}="3" & 
MM \ar@{=}[d] \\
M \ar[r]_-{m^*}^{~}="2" & MM \ar[r]_-{1a}^{~}="4" & MM 
\ar@{=>}"1";"2"^{a^2} 
\ar@{=>}"3";"4"^{a\delta} }$}
\end{equation}

\subsection{Modules}\label{sec:Galois}

Let $a$ be a bialgebra in $\cm(M,M)$ for a map-monoidale $M$ in a
monoidal bicategory $\cm$. Since, in particular, $a$ is a
convolution-monoid, we can define (right) actions of $a$ on
objects of $\cm(M,M)$. We define an {\em $a$-module} to be an object
$q\in\cm(M,M)$ equipped with an associative unital action $\gamma\colon
q\bullet a\to q$. Thus $a$-modules are the same as algebras for the monad
$-\bullet a$. Explicitly, the 2-cell $\gamma$ has the form displayed in the
diagram on the left below, 
$$\xymatrix{
M^2 \ar[r]^-{qa}_{~}="1" & M^2 \ar[d]^-{m} & M^2 \ar[r]^-{qa}_{~}="3" \ar[d]_-{m}
& 
M^2 \ar[d]^-{m} & M^2 \ar[r]^-{qa}_{~}="5"  & M^2 \\
M \ar[u]^-{m^*} \ar[r]_-{q}^{~}="2" & M & M \ar[r]_-{q}^{~}="4" & M & 
M \ar[u]^-{m^*} 
\ar[r]_-{q}^{~}="6" & M \ar[u]_-{m^*} 
\ar@{=>}"1";"2"^{\gamma} 
\ar@{=>}"3";"4"^{q_2} 
\ar@{=>}"5";"6"^{q^2} }$$
but pasting with the unit $\eta_m$ of the adjunction $m\dashv m^*$ allows this
to be expressed in terms either of a 2-cell $q_2$ or as a 2-cell $q^2$ as on
the right. 

In any case, the monad $-\bullet a$ is opmonoidal, thanks to the
bialgebra structure on $a$, and so the category of $a$-modules becomes
monoidal \cite{AguiarMahajan-monoidal}. Explicitly, the tensor product of 
$a$-modules $(q,\gamma)$ and $(q',\gamma')$ is $q\circ q'$ equipped with the 
action  
$$\xymatrix{
(q\circ q')\bullet a \ar[r]^-{1\bullet\delta} & 
(q\circ q')\bullet(a\circ a) \ar[r]^-{\xi} & 
(q\bullet a)\circ(q'\bullet a) \ar[r]^-{\gamma\circ\gamma'} & q\circ q'.}$$

For an $a$-module $(q,\gamma)$ there are morphisms $\beta_{q,x}\colon (q\circ
x)\bullet a\to q\circ(x\bullet a)$, natural in the object $x$ of $\cm(M,M)$,
and given by the  composite 
$$
\xymatrix @C1.5pc {
(q\circ x) \bullet a \ar[r]^-{1\bullet\delta} & 
(q\circ x) \bullet(a\circ a) \ar[r]^-{\xi} & 
(q\bullet a)\circ(x\bullet a) \ar[r]^-{\gamma\circ 1} & 
q\circ(x\bullet a)}
$$
or equivalently as 
\begin{equation}\label{eq:beta}
\xymatrix @C4pc {
m.xM.qa.m^* \ar[r]^-{m.xM.q\delta.m^*} & m.xa.qa.m^* \ar[r]^-{m.xa.q^2} & 
m.xa.m^*.q.}
\end{equation}
These maps are called the {\em Galois maps} of $a$. 

\subsection{Comodules}\label{sec:coGalois}

Dually to the previous section, for a map-monoidale $M$ in a monoidal
bicategory $\cm$, and a bialgebra $a$ in $\cm(M,M)$, we define a (right) {\em 
$a$-comodule} to be an object $p\in\cm(M,M)$ equipped with a coassociative
counital coaction $\rho\colon p\to p\circ a$; in other words, a coalgebra for
the comonad $-\circ a$; this time $\rho$ has the simpler form   
$$
\xymatrix{
M\ar[rr]^-p_{~}="1" \ar[rd]_-p & {}\ar@{=>}[d]^{\rho} &
M \\
& M\ar[ru]_-a\ar@{}^{~}="2"}
$$
in terms of \cm.  

Since this comonad is monoidal \cite{AguiarMahajan-monoidal}, the category of 
$a$-comodules is also monoidal, with the tensor product of $(p,\rho)$ and 
$(p',\rho')$ given by $p\bullet p'$ with coaction 
$$\xymatrix{
p\bullet p' \ar[r]^-{\rho\bullet\rho'} & 
(p\circ a)\bullet(p'\circ a) \ar[r]^-{\xi} & 
(p\bullet p')\circ(a\bullet a) \ar[r]^-{1\circ\mu} & 
(p\bullet p')\circ a .}$$

Once again there are maps $\zeta_{p,x}\colon p\bullet (x\circ a)\to (p\bullet
x)\circ a$, natural in  the object $x$ of $\cm(M,M)$, and this time given by   
$$
\xymatrix{
p\bullet(x\circ a)\ar[r]^-{\rho\bullet1} & 
(p\circ a)\bullet (x\circ a)\ar[r]^-{\xi} & 
(p\bullet x)\circ(a\bullet a) \ar[r]^-{1\circ\mu} & 
(p\bullet x)\circ a
}$$
or equivalently by 
\begin{equation}\label{eq:zeta}
\xymatrix @C4pc {
m.Ma.px.m^* \ar[r]^-{m.Ma.\rho x.m^*} & m.aa.px.m^* \ar[r]^-{a_2.px.m^*} & 
a.m.px.m^* }
\end{equation}
and these are called {\em co-Galois maps}.

\section{Duality}

\subsection{Duality principles for duoidal categories}

As observed in \cite[Section~4.3]{Street-chaire} and
\cite{AguiarMahajan-monoidal}, there are various dualities available for
duoidal categories. These are higher-dimensional analogues of the dualities
for {\em double algebras} described in \cite{Szlachanyi-double}. 

For any duoidal category one can obtain new duoidal categories by reversing
either or both of the monoidal structures. For any duoidal category \cd, we
write $\cd\rev$ for the duoidal category obtained from \cd by reversing
both. Thus if we write $f\rev$ for an object $f\in\cd$, seen as lying in
$\cd\rev$, then $f\rev\circ g\rev=(g\circ f)\rev$ and $f\rev\bullet
g\rev=(g\bullet f)\rev$.  

We can also obtain a duoidal structure on $\cd\op$. If we write $f\op$ for an
object $f\in\cd$, seen as lying in $\cd\op$, then $f\op\circ g\op=(f\bullet
g)\op$ and $f\op\bullet g\op=(f\circ g)\op$. 

\subsection{Duality in monoidal bicategories} \label{sect:bicategorical-dual}

Let $X$ be an object of the monoidal bicategory \cm. A {\em right dual for
  $X$} consists of an object $\dual{X}$ equipped with morphisms $n\colon
I\to\dual{X}X$ and $e\colon X\dual{X}\to I$ satisfying the triangle equations
up to coherent isomorphism \cite{DS:Hopf-algebroid}. 

Let $\cm\du$ be the full sub-bicategory of \cm consisting of those objects with
right duals; this is in fact closed under the monoidal structure, with
$\dual{XY}$ naturally isomorphic to $\dual{Y}\,\dual{X}$ and $I$ self-dual. 
Write $\cm\oprevdl$ for $((\cm\op)\du)\rev=((\cm\rev)\du)\op$; this has
objects the objects of \cm with {\em left} duals. There is a monoidal
biequivalence $\cm\du\sim \cm\oprevdl$ of monoidal bicategories sending an
object $X$ to $\dual{X}$ \cite{DS:Hopf-algebroid}. A morphism $f\colon X\to Y$
is sent to the composite
$$\xymatrix{ 
\dual{Y} \ar[r]^-{n1} & 
\dual{X}X\dual{Y} \ar[r]^-{1f1} & 
\dual{X}Y\dual{Y} \ar[r]^-{1e} & \dual{X} }
$$
which we call $f^+$. The inverse sends $g\colon\dual{Y}\to\dual{X}$ to $g^-$
defined by  
$$\xymatrix{
X \ar[r]^-{1n} & 
X\dual{Y}Y \ar[r]^-{1g1} & 
X\dual{X}Y \ar[r]^-{e1} & Y . }
$$

In particular, for any object $X\in\cm$ with a right dual $\dual X$, we have 
a monoidal equivalence $\cm(X,X)\simeq\cm(\dual{X},\dual{X})\rev$.

\subsection{Duality and map-monoidales}\label{sec:duality&map-monoidales}

Of course a monoidal biequivalence preserves (in an up-to-equivalence sense)
any structure expressible in a monoidal bicategory, such as map-monoidales,
morphisms between them, and composition and convolution products.  

Thus if $M$ is a map-monoidale, which as an object of \cm has a right dual
$\dual{M}$, then it is a map-monoidale in $\cm\du$, and so
$\dual{M}$ is a map-monoidale in $\cm\oprevdl$, and the induced
equivalence $\cm\du(M,M)\simeq \cm\oprevdl(\dual{M},\dual{M})$ is a strong
duoidal equivalence. (Recall that a functor between duoidal categories is {\em
strong duoidal}, or {\em 2-strong monoidal} in the original nomenclature of
\cite{AguiarMahajan-monoidal}, if it preserves all the duoidal structure up to
coherent natural isomorphism; this means in particular that it is strong
monoidal with respect to both monoidal structures, but also that these
isomorphisms are compatible with the structure maps $\xi$, $\xi^0$, $\xi_0$,
and $\xi^0_0$.) Since $\cm\du$ is a full sub-bicategory of \cm, we may write
this strong duoidal equivalence more simply as
$\cm(M,M)\simeq\cm(\dual{M},\dual{M})\rev$. 

Recall that if the map-monoidale $(M,m,u)$ is naturally Frobenius, the object
$M$ is self-dual in the monoidal bicategory \cm, with unit and counit 
$$\xymatrix{ I \ar[r]^-{u} & M \ar[r]^-{m^*} & MM
&&
MM \ar[r]^-{m} & M \ar[r]^-{u^*} & I . }
$$
Thus a morphism $f\colon M\to M$ has mates $f^+$ and $f^-$ given by  
$$
\xymatrix @R 0pc {
M \ar[r]^-{u1} & M^2 \ar[r]^-{m^*1} & M^3 \ar[r]^-{1f1} & M^3 \ar[r]^-{1m} & 
M^2 \ar[r]^-{1u^*} & M \\
M \ar[r]^-{1u} & M^2 \ar[r]^-{1m^*} & M^3 \ar[r]^-{1f1} & M^3 \ar[r]^-{m1} & 
M^2 \ar[r]^-{u^*1} & M }
$$
and these assignments are mutually inverse, in the sense that $(f^-)^+ \cong f
\cong (f^+)^-$. 

These form part of a duoidal equivalence $\cm(M,M)\simeq\cm(M,M)\rev$, thanks 
to the monoidal biequivalence $\cm\du\simeq{\cm\oprevdl}$ of
Section~\ref{sect:bicategorical-dual}. We shall need notation for the
structure maps. In the case of the composition structure, we write
$\Xi=\Xi_{f,g}\colon g^-\circ f^-\cong (f\circ g)^-$ and $\Xi_0\colon i\cong
i^-$ for the structure maps. For the convolution structure we write
$\Upsilon=\Upsilon_{f,g}\colon g^-\bullet f^-\cong(f\bullet g)^-$ and
$\Upsilon_0\colon j\cong j^-$. Their explicit forms can be found in
Appendix \ref{app:Xi-Upsilon}. 

For a naturally Frobenius map-monoidale $(M,m,u)$ in a monoidal bicategory
\cm, consider the naturally Frobenius map-monoidale $(M,m^*,u^*)$ in $\cm
\oprev$. The monoidal biequivalence $(-)^-:{\cm\oprevdl}\simeq \cm\du$ of
Section~\ref{sect:bicategorical-dual} takes it to the naturally Frobenius
map-monoidale $(M,m^{*-},u^{*-})$ in \cm, but the identity morphism $1\colon
M\to M$ underlies a  monoidal equivalence $(M,m^{*-},u^{*-})\simeq(M,m,u)$,
and we generally identify these monoidales. We shall write $\chi\colon
m^{*-}\to m$ for the isomorphism involved in this monoidal equivalence.  

\begin{remark}
Some double algebras, in the sense of \cite{Szlachanyi-double}, possess an
endomorphism $S$ called an antipode, and defined equationally. In a duoidal
category arising from a naturally Frobenius map-monoidale, the functor $S$
sending $f$ to $f^-$ satisfies these ``antipode axioms'' up to natural
isomorphism. 

In the particular class of double algebras, obtained in
\cite[Section 8.5]{Szlachanyi-double} as endomorphism algebras of Frobenius
extensions, the explicit expressions of $S$ and $S^{-1}$ are direct analogues 
of our formulae for $(-)^-$ and $(-)^+$.
\end{remark}

\subsection{Duality for monoidal comonads}\label{sect:minus}

In light of the duoidal equivalence between $\cm(M,M)$ and $\cm(M,M)\rev$,
if $a$ is a monoidal comonad on a naturally Frobenius map-monoidale
$M$, then $a^-$ also has a monoidal comonad structure on $M$.
 
This construction will play a crucial role in our analysis of antipodes. For a
Hopf algebra $H$, the antipode can be seen as a coalgebra homomorphism from
$H$ to the coalgebra $H\op$  obtained from $H$ by using the
reversed comultiplication (and, likewise, as an algebra homomorphism). In our
context, the antipode will have the form of a morphism $a\to a^-$ of
bialgebras. This time, however, even if we are not interested in the
preservation of bialgebra structure we are still forced to work with $a^-$,
since there is no analogue of the fact that the Hopf algebras $H$ and $H\op$
have the same underlying vector space. 

The Hopf map $\hat{\beta}$ for the monoidal comonad $a^-$ is in fact the
co-Hopf map $\hat{\zeta}$ for $a$; more precisely, there is a
commutative diagram  
\begin{equation}\label{eq:beta-zeta_duality}
\xymatrix{
m.a^- a^- \ar[d]_{\hat{\beta}} \ar[r] &
m^{*-}.a^-a^- \ar[r]   &  m^{*-}.(aa)^- \ar[r]  &
(aa.m^*)^-\ar[d]^-{\hat \zeta ^-} \\
a^-.m.a^-1 \ar[r] &
 a^-.m^{*-}.a^-1^- \ar[r] &  a^-.m^{*-}.(1a)^- \ar[r]  &
(1a.m^*.a)^-}
\end{equation}
where the un-named arrows are isomorphisms arising from $\chi\colon m^{*-}\to
m$ and the various preservation properties of the monoidal biequivalence
$\cm\oprevdl\sim\cm\du$.  

\subsection{Further structure in the naturally Frobenius case} 
\label{sec:more-on-NatFrob}

For any two morphisms $f\colon M\to M$ and $g\colon M\to M$, there is a 2-cell 
$$
\phi_{f,g}\colon f\circ g^-\to ((f\bullet g)\circ j)\bullet i  
$$
natural in $f$ and $g$, and given by the following pasting composite.
$$
\xymatrix{
& M \ar[r]^-{1u} & M^2 \ar[r]^-{1m^*} & M^3 \ar[r]^-{1g1} & 
M^3 \ar[d]^-{m1} \\
M \ar[ur]^-{f}  \ar[r]^-{1u} \ar@{=}@/_2pc/[drr] &
M^2 \ar@{=}[r]_(0.79){~}="1" \ar[dr]_-{m}^(0.7){~}="2"  & 
M^2 \ar[r]^-{1m^*} & 
M^3 \ar[u]^-{f11} & M^2 \ar[r]^-{u^*1} \ar[dr]_-{j1} & 
M \ar[d]^-{u1} \ar@{=}[dr] \\
&& M \ar[u]_-{m^*} \ar[r]^-{m^*} 
\ar@/_2pc/[rrrr]_-{((f\bullet g)\circ j)\bullet i} & 
M^2 \ar[u]^-{m^*1} \ar[ur]_-{(f\bullet g)1} 
\ar[rr]_-{((f\bullet g)\circ j)1} && 
M^2 \ar[r]^-{m} & M 
\ar@{=>}"1";"2"_{\eta_m} 
}
$$

We can play the same game when we consider the naturally Frobenius
map-monoid\-ale  $(M,m,u)^*=(M,m^*,u^*)$ in $\cm\oprev$. Since
this interchanges $m$ and $m^*$, and $u$ and $u^*$, as well as reversing the
order of composition and the order of tensoring, the morphism $f^-$ defined
above does not depend on whether we work with $(M,m,u)$ or $(M,m,u)^*$. On the
other hand, the maps $\phi_{f,g}$ do so depend: the morphism $\phi_{f,g}$
defined using $(M,m,u)^*$ is a morphism  
$$
\psi_{f,g}\colon  f^-\circ g\to i\bullet(j\circ(f\bullet g))
$$
in $\cm(M,M)$, constructed in a dual manner to that given above for $\phi$. 

\begin{lemma}\label{lem:phi-psi-identities}
For the 2-cells $\phi$ and $\psi$ above, and for any 1-cells $f,g,h,k:M\to M$,
the following diagrams commute. 
$$ 
\xymatrix@R=5pt @C=45pt{ 
(f^-\circ g)\bullet (h^-\circ k)\ar[r]^-{\psi_{f,g}\bullet 1}
\ar[ddd]_-\xi &
i\bullet (j\circ (f\bullet g)) \bullet (h^-\circ k)
\ar[dd]^-{1\bullet \xi}
\\
\\
& i\bullet (h^-\circ(f\bullet g \bullet k)) 
\ar[dd]^-{1 \bullet \psi_{h,f\bullet g \bullet k}} 
\\
(f^-\bullet h^-) \circ (g\bullet k) \ar[ddd]_-{\Upsilon_{h,f}\circ1} &
\\
& i\bullet i \bullet (j\circ(h\bullet f \bullet g\bullet k))
\ar[dd]^-{\xi_0 \bullet 1}
\\
\\
(h\bullet f)^- \circ (g\bullet k) \ar[r]_-{\psi_{h\bullet f,g\bullet k}}&
i\bullet (j\circ(h\bullet f \bullet g\bullet k)) }
$$
$$
\xymatrix@R=5pt @C=45pt{ 
(f\circ g^-)\bullet (h\circ k^-) \ar[r]^-{1\bullet \phi_{h,k}}
\ar[ddd]_-\xi&
(f\circ g^-)\bullet ((h\bullet k)\circ j) \bullet i 
\ar[dd]^-{\xi \bullet 1}
\\
\\
&((f\bullet h\bullet k)\circ g^-)\bullet i
\ar[dd]^-{\phi_{f\bullet h \bullet k,g} \bullet 1}
\\
(f\bullet h)\circ (g^-\bullet k^-)\ar[ddd]_-{1\circ\Upsilon_{k,g}} & 
\\
& ((f\bullet h \bullet k\bullet g)\circ j) \bullet i\bullet i
\ar[dd]^-{1\bullet\xi_0}
\\
\\
(f\bullet h)\circ (k\bullet g)^- \ar[r]_-{\phi_{f\bullet h,k\bullet g}}&
((f\bullet h \bullet k\bullet g)\circ j) \bullet i  .}
$$
\end{lemma}

\proof
In order to see commutativity of the first diagram, use the explicit forms of
$\psi$, $\Upsilon$, and  $\xi$;
unitality and associativity of $m:M^2\to M$; a triangle identity on the
adjunction $m\dashv m^*$; and pseudo-naturality of the occurring 2-cells.
Commutativity of the second diagram follows symmetrically.
\endproof

\begin{lemma} \label{lem:theta}
For any 1-cells $f,g,h:M\to M$, there is a 2-cell
$$
\vartheta_{f,g,h}:f\circ ((g\circ j)\bullet i)\circ h^- \to  
((f\bullet h)\circ g\circ j)\bullet i
$$  
obeying the following properties.
\begin{itemize}
\item[{(i)}] $\vartheta$ is natural in each of the 1-cells $f,g,h$.
\item[{(ii)}] For any 1-cells $f$ and $h$ the equality
$\vartheta_{f,i,h}=\phi_{f,h}$ holds (modulo the isomorphisms $i\circ j\cong
j$ and $j\bullet i\cong i$). 
\item[{(iii)}] For any 1-cells $f,g,h,k$, the following diagram commutes. 
$$
\xymatrix{
f\circ g\circ h^-\circ k^-\ar[r]^-{1\circ \phi_{g,h} \circ 1} 
\ar[d]_-{1\circ 1\circ\Xi_{k,h}} &
f\circ ((g\bullet h)\circ j)\bullet i) \circ k^- 
\ar[dd]^-{\vartheta_{f,g\bullet h,k}}\\
f\circ g \circ (k\circ h)^-\ar[d]_-{\phi_{f\circ g,k\circ h}}\\
(((f\circ g)\bullet (k\circ h))\circ j)\bullet i
\ar[r]_-{(\xi\circ 1)\bullet 1}&
((f\bullet k)\circ (g\bullet h)\circ j)\bullet i }
$$
\item[{(iv)}] For any 1-cells $f,g,h$, the following diagram commutes.
$$
\xymatrix{
f \circ ((g\circ j)\bullet i)\circ h^- \ar[r]^-{\vartheta_{f,g,h}}
\ar[d]_-{1 \circ ((1\circ \xi^0)\bullet 1)\circ 1} &
((f\bullet h)\circ g \circ j)\bullet i 
\ar[d]^-{(1 \circ 1 \circ \xi^0)\bullet 1 } \\
f \circ ((g\circ j\circ j)\bullet i)\circ h^- 
\ar[r]_-{\vartheta_{f,g\circ j,h}} &
((f\bullet h)\circ g \circ j\circ j)\bullet i }
$$
\end{itemize}
\end{lemma}

\proof 
We construct $\vartheta_{f,g,h}$ as the pasting composite
$$
\xymatrix{
M \ar@{=}@/^1.5pc/[rr] \ar[r]_{m^*} \ar[dd]^{m^*} & 
M^2 \ar[r]^{1u^*}_(0.5){~}="1" \ar@{=}@/_1pc/[dr]^(0.59){~}="2" 
\ar@{=>}"1";"2"^{1\epsilon_u}  & 
M \ar[r]^f \ar[d]^{1u} &
M \ar[r]^{m^*} & 
M^2 \ar[r]^{g1} \ar[dddrr]^{1h^-} \ar[d]_{11u} & 
M^2 \ar[r]^{u^*1} & 
M  \ar[r]^-{u1}\ar@{=}@/_1pc/[rr] &
M^2\ar[r]^-{m} &
M \ar[d]^{h^-}\\
&&
M^2 \ar[d]^{1m^*} &&
M^3 \ar[d]_{11m^*} &&&&
M \ar[d]_-{u1}\ar@{=}@/^1pc/[dd]
\\
M^2 \ar[rr]^{m^*1} \ar@/_1pc/[rrrrdd]_{(f\bullet h)1} 
&&
M^3 \ar[r]^{f11} \ar[dr]_{fh1} & 
M^3 \ar[r]^{m^*11} \ar[d]^{1h1} &
M^4 \ar[d]^{11h1} &&&
M \ar[ru]^{u^*1} &
M^2 \ar[d]_-{m}
\\
&&& 
M^3 \ar[r]^{m^*11} \ar[dr]_{m1}  & 
M^4 \ar[r]^{1m1} \ar@{=>}[d]^(0.4){\pi'1} &
M^3 \ar[r]^{1u^*1} &
M^2 \ar[ru]^{g1} & & 
M\\
&&&&
M^2 \ar[ur]_{m^*1} \ar@{=}@/_1pc/[urr] }
$$
where the undecorated regions denote the associativity and the
unit constraints of $m$; the counitality and coassociativity 
constraints of $m^*$, an identity 2-cell from the definition of
$h^-$, and some middle-four interchange laws in $\cm$. Assertions (i), (ii),
and (iv) are immediate by the construction of $\vartheta$. Part (iii) follows 
by the explicit forms of $\Xi$, $\psi$, $\xi$, and $\vartheta$, using a
triangle identity on the adjunction $u\dashv u^*$, unitality and associativity
of $m$, and pseudo-naturality.
\endproof

\section{Examples}\label{sec:examples}

\subsection{Bialgebras in braided monoidal categories}

A monoidal category $\cc$ can be regarded as a bicategory $\cm$ with a
single object $\ast$. The hom category $\cm(\ast,\ast)$ is $\cc$
and the horizontal composition is provided by the  reverse monoidal
product $\ox\rev$ of $\cc$. Now if $\cc$ is in addition braided,
then this bicategory $\cm$ is monoidal via the
monoidal product also given by $\ox$. The interchange law, between the
horizontal composition $\ox\rev$ and the monoidal product $\ox$, is provided
by the braiding $\beta$ in $\cc$ as   
$$
\xymatrix{
(x\ox\rev y)\ox (z \ox\rev v) =
y\ox x \ox v \ox z \ar@{|->}[r]^-{1\ox \beta \ox 1} &
y\ox v \ox x\ox z  =
(x \ox  z) \ox\rev  (y  \ox  v).
}$$
Clearly, the single object $\ast$ is a trivial naturally Frobenius
map-monoidale in $\cm$ rendering $\cc\cong \cm(\ast,\ast)$ a duoidal category. This is the duoidal structure discussed
in \cite[Section 6.3]{AguiarMahajan-monoidal}: both the composition product
$\circ$ and the convolution product $\bullet$ are equal to $\ox$. We conclude
that this duoidal category arises from a suitable naturally Frobenius
map-monoidale in a monoidal bicategory. Thus we obtain the following. 

\begin{example}\label{ex:braided_bialg}
Regard a braided monoidal category $\cc$ as a monoidal bicategory $\cm$ with a single object. Monoidal comonads in $\cm$ on the
trivial naturally Frobenius map-monoidale are the usual bialgebras in the
braided monoidal category $\cc$.
\end{example}

\subsection{Bialgebroids}

In our next example we take \cm to be the monoidal bicategory
\Mod: an object of \Mod is a ring, a morphism is a bimodule, and a 2-cell is a 
homomorphism of bimodules. Morphisms $R\to S$ and $S\to T$ are composed by
tensoring over $S$; the monoidal structure is given by the usual tensor
product $\ox$ of rings (and of modules and their
homomorphisms). The unit object $I$ is $\bz$ (or the base ring, if one is
working over some other commutative ring).

If $R$ is a {\em commutative} ring, then the multiplication $R\ox R\to R$ is a
homomorphism of rings; of course the unit also determines a ring homomorphism
$I\to R$, and so one has a map-monoidale in \Mod.  

We now analyze what being naturally Frobenius means in this case. The
multiplication $m\colon RR\to R$ is $R$, seen as a left $R\ox R$, right
$R$-module; all the actions are regular. We write this as
$
{}_{\bullet\bullet}R_\bullet$ or sometimes
$
{}_{ab}R_c$. The adjoint $m^*$ is then
$
{}_{\bullet}R_{\bullet\bullet}$, and so the composite $m^*.m$ is
given by  
$$\tensor[_a_b]{R}{_x} \tensor[_x]\ox{_x} \tensor[_x]{R}{_c_d} \cong
\tensor[_a_b]{R}{_c_d} \cong \tensor[_a_c]{R}{_b_d}$$ 
where the last step uses commutativity of $R$ to allow left and right actions
to be interchanged. On the other hand the composite $1m.m^*1$ is given by 
\begin{align*}
\tensor[_a]{R}{_x_y} \ox \tensor[_b]{R}{_z} \tensor[_x_y_z]\ox{_x_y_z} 
\tensor[_x]{R}{_c} \ox \tensor[_y_z]{R}{_d}  &\cong 
\tensor[_a]{R}{_c_y} \ox\tensor[_b]{R}{_z} \tensor[_y_z]{\ox}{_y_z} 
\tensor[_y_z]{R}{_d}  \\
&\cong \tensor[_a]{R}{_c_y} \tensor[_y]{\ox}{_y} \tensor[_y_b]{R}{_d} \\
&\cong \tensor[_a_c]{R}{_y} \tensor[_y]{\ox}{_y} \tensor[_y]{R}{_b_d} \\
&\cong \tensor[_a_c]{R}{_b_d}
\end{align*}
and so $1m.m^*1\cong m^*.m$; one can check that the composite
isomorphism we have constructed is indeed $\pi'$. 
Similarly $\pi\colon m1.1m^*\to m^*.m$ is invertible. 

Thus any commutative ring $R$ determines a naturally Frobenius 
map-monoidale in \Mod, giving rise to a duoidal category as follows. 

\begin{example}\label{ex:bialgebroid}
The hom-category $\Mod(R,R)$ is the category of $(R,R)$-bimodules. The
$\circ$-tensor is given by tensoring over $R$, and $i$ is the the regular
bimodule $R$. Since $R$ is commutative, $(R,R)$-bimodules can be regarded as
$R\ox R$-modules, and $R\ox R$ is itself commutative, thus tensoring over
$R\ox R$ defines the second monoidal structure $\bullet$ on $\Mod(R,R)$ with
unit $R\ox R$. This duoidal category was studied in
\cite[Example~6.18]{AguiarMahajan-monoidal}. A bimonoid in this duoidal
category is precisely an $R$-bialgebroid $A$ for which the maps $s,t\colon
R\to A$ land in the centre of $A$: see \cite[Example
6.44]{AguiarMahajan-monoidal} and \cite[Section~4.3]{BohmChenZhang}.  
\end{example}

\subsection{Weak bialgebras}

Our next example also involves the monoidal bicategory \Mod, but the naturally
Frobenius monoidale will be of a different type.  

If $R$ is a ring, we may of course regard it as a right $R\op\ox R$-module,
and so as a morphism $n\colon I\to R\op R$ in \Mod. Similarly we may regard it
as a left $R\ox R\op$-module and so as a morphism $e\colon RR\op\to I$ in
\Mod. These satisfy the triangle equations for an adjunction with $n$ as the
unit and $e$ the counit. It follows that $R\op R$ becomes a monoidale in \Mod
with multiplication  
$$\xymatrix{
R\op RR\op R \ar[r]^-{1e1} & R\op R }$$
and unit $n$.

In general, of course, $e$ and $n$ are not maps, but they are so when $R$ is
separable Frobenius; that is, its multiplication has an $R$-bimodule
section $R\to R \ox R$ which is in addition a counital
comultiplication. So in this case we obtain a duoidal category as
follows. 

\begin{example}\label{ex:weak-bialgebra}
The category $\Mod(R\op\ox R, R\op\ox R)$ is duoidal for a separable
Fro\-benius algebra $R$, with both the $\circ$-product and the
$\bullet$-product  given by tensoring over $R\op\ox R$; though built on
different actions in both cases. This duoidal category was studied in
\cite{BohmSpanish-CatWkBialg}, where it was shown that a bimonoid therein was
the same as a weak bialgebra whose separable Frobenius base algebra is
isomorphic to $R$.
\end{example}

\subsection{Categories}

Our next example involves the monoidal bicategory $\Span$. An object of \Span
is a set, a morphism from $X$ to $Y$ is a {\em span} $(u,E,v)$ from $X$ to
$Y$, consisting of a set $E$ equipped with functions $u\colon E\to X$ and
$v\colon E\to Y$. These are composed via pullback. A 2-cell in \Span from
$E$ to $F$ is a function from $E$ to $F$, commuting with the maps into $X$ and
$Y$.  

Any  function $f\colon X\to Y$ determines a span $f_*=(1,X,f)$ from $X$ to
$Y$. Such a span has a right adjoint $f^*=(f,X,1)$ from $Y$ to $X$;
furthermore, every left adjoint in \Span is isomorphic to one of the form
$f_*$.  

The cartesian product of sets makes \Span into a monoidal bicategory (but the
tensor product is not the product in \Span). 

Every set $X$ has a unique comonoid structure in \Set, obtained using the
diagonal $\Delta\colon X\to X\x X$ and the unique map $X\to 1$. Now $\Delta^*$
makes $X$ into a monoidale in \Span. It fails to be a map-monoidale since
$\Delta^*$ is a right adjoint rather than a left adjoint.  
We fix this by moving from \Span to $\Span\co$, in which the 2-cells are
formally reversed; thus the left adjoints in $\Span\co$ are the right adjoints
in \Span. In conclusion, every set $X$ is a map-monoidale in $\Span\co$. 

Furthermore, these map-monoidales are naturally Frobenius; the isomorphisms
$m1.1m^*\cong m^*.m\cong 1m.m^*1$ essentially amount to the fact that the
square 
$$\xymatrix @R1pc @C1pc { 
& X \ar[dr]^-{\Delta} \ar[dl]_-{\Delta} \\
XX \ar[dr]_-{\Delta1} && XX \ar[dl]^-{1\Delta} \\
& XXX }$$
is a pullback in \Set. 

Thus any set $X$ determines a naturally Frobenius map-monoidale in $\Span\co$,
giving rise to a duoidal category as follows. 

\begin{example}\label{ex:category}
The hom-category $\Span\co(X,X)$ is by definition $\Span(X,X)\op$, which in
turn is the opposite  $(\Set/X\x X)\op$ of the slice category $\Set/X\x
X$. The convolution product $\bullet$ is just the product in $\Set/X\x X$,
given by pulling back morphisms into $X\x X$. Every object has a unique
$\bullet$-monoid structure, and every morphism is a homomorphism of
$\bullet$-monoids. The unit $j$ is $X\x X$. The other tensor product $\circ$
is also defined by a pullback, as in the following diagram  
$$\xymatrix @R1pc @C1pc {
&& E\circ F \ar[dr] \ar[dl] \\
& E \ar[dl] \ar[dr] && F \ar[dl] \ar[dr] \\
X && X && X. }$$
A $\circ$-comonoid is precisely a category with object-set $X$; since
$\bullet$-monoid structure is automatic, the bimonoids are also just the
categories with object-set $X$: see \cite[Examples 6.17 and
6.43]{AguiarMahajan-monoidal} or \cite[Section~4.2]{BohmChenZhang}. 
\end{example}

\subsection{Monoidal comonads on autonomous monoidal categories} 

The bicategory \Prof has categories as objects, profunctors $A\to B$
(also known as distributors or modules) as morphisms, and natural
transformations as 2-cells. Recall that a profunctor form $A$ to $B$ is a
functor $B\op\x A\to\Set$, and that the composite of profunctors $f\colon A\to
B$ and $g\colon B\to C$ is given by the coend $(g\circ f)(c,a)=\int^{b\in B}
g(c,b)\x f(b,a)$. Recall further that every functor $f\colon A\to B$ gives
rise to a profunctor, called $f_*$ or just $f$, given by $f_*(b,a)=B(b,fa)$,
and that this has an adjoint $f\dashv f^*$, given by $f^*(a,b)=B(fa,b)$.  In
fact, these constructions are the object maps of pseudofunctors $(-)_*:\Cat\to
\Prof$ and $(-)^*:\Cat\coop\to \Prof$, respectively. 

The bicategory \Prof is monoidal, with tensor product being the
cartesian product of categories (but the resulting monoidal structure on \Prof
is not itself  cartesian). A
monoidale in \Prof is a promonoidal category in the sense of Day
\cite{Day-convolution}, while a map-monoidale is essentially just a monoidal
category. The monoidal category is naturally Frobenius, as a map-monoidale in
\Prof, just when it has left and right duals: see
\cite[Theorem~6.4]{Nacho-FormalHopfAlgebraI} or 
\cite[Remark~6.3]{dualsinvert}. There are also enriched variants of this 
example; see \cite{Nacho-FormalHopfAlgebraI} once again. 

If $M$ is a monoidal category with left and right duals, $N$ is another
monoidal category, and $f\colon N\to M$ a strong monoidal functor, then the
functor $f$ has an adjoint $f\dashv f^*$ in \Prof, and the induced comonad
$ff^*$ is monoidal in \Prof, and so it can be regarded as a bimonoid in 
the duoidal category  $\Prof(M,M)$.

In particular, if a strong monoidal functor $f$  has a right
adjoint in \Cat, then it induces  a monoidal comonad in \Cat on $M$; thus it
gives rise to a monoidal comonad in \Prof. 
If a functor $f$ not only has a right adjoint but is comonadic,
then to say that $f$ is strong monoidal is equivalent to saying that the
induced comonad is monoidal. 

We record this as:

\begin{example}\label{ex:Prof}
If $M$ is a monoidal category with left and right duals, then the category 
$\Prof(M,M)$ of profunctors from $M$ to $M$ is duoidal. Any monoidal comonad 
on $M$ can be seen as a bimonoid in $\Prof(M,M)$. 
\end{example}

\section{Transforms}

In this section we describe a ``transform'' process relating two
isomorphic categories. It is analogous to the isomorphism between the algebra
of $H$-module and $H$-comodule homomorphisms $H\ox H \to H\ox H$, and 
the convolution algebra $\End(H)$, for a Hopf algebra $H$. It will play a key
role in our treatment of antipodes in the following section. 

We suppose throughout this section that $M=(M,m,u)$ is a naturally Frobenius
map-monoidale in the monoidal bicategory \cm, that $(b,\mu,\eta)$ is a monoid
with respect to the convolution $\bullet$, and that $(c,\delta,\varepsilon)$
is a comonoid with respect to the composition $\circ$ in the duoidal category
$\cm(M,M)$.

\begin{claim}[{\bf A category of mixed algebras}] \label{sec:BTG} 
Let \cb be the category $\cm(M^2,M)$ of all morphisms from $M^2$ to $M$. 
There is an induced comonad $\cm(cM,M)$ on \cb sending
$x\colon M^2\to M$ to $x.cM$. We call this comonad $G$, and
write $\cb^{G}$ for the category of $G$-coalgebras.  

There is also a monad $T$ on \cb sending $x\colon M^2\to M$ to the composite 
$$\xymatrix{
M^2 \ar[r]^-{Mm^*} & M^3 \ar[r]^-{xM} & M^2 \ar[r]^-{Mb} & 
M^2 \ar[r]^-{m} & M}$$
and with multiplication and unit given by the 2-cells
$$\xymatrix{
M^3 \ar[r]^-{Mm^*M} & M^4 \ar[r]^-{xM^2} & M^3 \ar[r]^-{Mbb}_{~}="1" & 
M^3 \ar[r]^-{mM} \ar[d]^-{Mm} & M^2 \ar[d]^-{m} \\
M^2 \ar[u]^-{Mm^*} \ar[r]_-{Mm^*} & M^3 \ar[u]^-{M^2m^*} \ar[r]_-{xM} & 
M^2 \ar[u]^-{Mm^*} \ar[r]_-{M b }^{~}="2"  & M^2 \ar[r]_-{m} & M 
\ar@{=>}"1";"2"^{M\mu} }$$
$$\xymatrix{ 
& M^2 \ar[r]^-{x} & M \ar@{=}[r]_{~}="1" & 
M \ar@{=}@/^1.2pc/[rd] \ar[d]^-{Mu} & 
\\
M^2 \ar@{=}@/^1.2pc/[ur] \ar[r]_-{Mm^*} & M^3 \ar[u]^-{M^2u^*} 
\ar[r]_-{xM} & 
M^2 \ar[u]^-{Mu^*} \ar[r]_-{Mb}^{~}="2"  & M^2 \ar[r]_-{m} & M  .
\ar@{=>}"1";"2"^{M\eta} }$$
We write $\cb^{T}$ for the category of algebras for this monad. An algebra
structure on $x$ translates to a 2-cell  
$$\xymatrix{
M^3 \ar[r]^-{xb}_{~}="1" \ar[d]_-{Mm} & M^2 \ar[d]^-{m} \\ 
M^2 \ar[r]_-{x}^{~}="2" & M \ar@{=>}"1";"2" }$$
satisfying associativity and unit conditions.

By functoriality of the tensor in \cm, there is an isomorphism $TG\cong GT$ of
functors, and a straightforward calculation shows that this defines a mixed
distributive law between the monad $T$ and comonad $G$. We write $\cb^{(T,G)}$
for the category of mixed algebras with respect to this distributive law. By 
the general theory of mixed distributive laws, $G$ lifts to a comonad on 
$\cb^{T}$ whose category of coalgebras is $\cb^{(T,G)}$, and $T$ lifts to a 
monad on $\cb^{G}$ whose category of algebras is $\cb^{(T,G)}$.  

In particular, let $x$ be the object $Mu^*\colon M^2\to M$. Then $GTx$ is the
composite   
$$\xymatrix{
M^2 \ar[r]^-{c  M} & M^2 \ar[r]^-{Mm^*} & M^3 \ar[r]^-{Mu^*M} & 
M^2 \ar[r]^-{Mb } &
M^2 \ar[r]^-{m} & M }$$
which, by counitality of $m^*$, is isomorphic to $m.cb \colon
M^2\to M$.  

Now let $y$ be the object $u.u^*.m\colon M^2\to M$. Then $GTy$ is the upper
composite in the diagram 
$$\xymatrix{
& M^3 \ar[r]^-{mM} & M^2 \ar[r]^-{u^*M} & M \ar[r]^-{uM} \ar[dr]_-{b } & 
M^2 \ar[r]^-{Mb } & M^2 \ar[d]^-{m} \\
M^2 \ar[r]_-{c  M} & M^2 \ar[u]^-{Mm^*} \ar[r]_-{m} & 
M \ar[u]^-{m^*} \ar@{=}[ur] && 
M \ar[ur]^-{uM} \ar@{=}[r] & M 
}$$
but using pseudofunctoriality of tensor in \cm, unitality of $m$, counitality
of $m^*$, and one of the Frobenius isomorphisms, we see that this is in fact
isomorphic to $b .m.cM$.  
\end{claim}

\begin{proposition}\label{prop:transform}
The full subcategory of $\cb^{(T,G)}$, determined by the two objects $GTx$ and
$GTy$ of Construction \ref{sec:BTG}, is isomorphic to a category
$\ct=\ct^c_b$ with objects $X$ and $Y$ in 
which:
\begin{itemize}
\item a morphism $X\to X$ is a morphism $c  \to i\bullet(j \circ b)$ in
  $\cm(M,M)$   
\item a morphism $X\to Y$ is a morphism $c  \to b $ in $\cm(M,M)$
\item a morphism $Y\to X$ is a morphism $c  \to b^-$ in $\cm(M,M)$   
\item a morphism $Y\to Y$ is a morphism $c  \to (b\circ j)\bullet i$ in
  $\cm(M,M)$.  
\end{itemize}
The identity on $X$ is given by 
$$\xymatrix{
c  \ar[r]^-{\epsilon} & 
i \ar@{=}[r] & i\bullet j \ar[r]^-{i\bullet \xi^0} & 
i\bullet(j\circ j) \ar[r]^-{i\bullet (j\circ\eta)} & 
i\bullet (j  \circ b) }$$
and the identity on $Y$ by 
$$\xymatrix{
c  \ar[r]^-{\epsilon} & 
i \ar@{=}[r] & j\bullet i \ar[r]^-{\xi^0\bullet i} & 
(j\circ j)\bullet i \ar[r]^-{(\eta\circ j )\bullet i} & 
(b\circ j ) \bullet i . }$$
The composites of $\sigma\colon X\to Y$ and $\tau\colon Y\to X$ are given by 
$$
\xymatrix{
c  \ar[r]^-{\delta} & 
c \circ c  \ar[r]^-{\sigma\circ\tau} & 
b \circ b ^- \ar[r]^-{\phi_{b ,b }} & 
((b\bullet b)\circ j)\bullet i\ar[r]^-{(\mu\circ j)\bullet i} &  
(b\circ j)\bullet i
}$$
and
$$\xymatrix{
c  \ar[r]^-{\delta} & 
c \circ c  \ar[r]^-{\tau\circ\sigma} & 
b ^-\circ b  \ar[r]^-{\psi_{b ,b }} & 
i\bullet  (j\circ (b\bullet b) ) \ar[r]^-{i\bullet(j\circ\mu)} & 
i\bullet (j\circ b) . }$$
\end{proposition}

\proof 
The hom-sets of the full subcategory of $\cb^{(T,G)}$  in question 
each have the form $\cb^{(T,G)}(GTw,GTz)$ for suitable $w$ and $z$. By the
universal property of the cofree coalgebra $GTz$, this is isomorphic to
$\cb^T(GTw,Tz)$; but $GTw\cong TGw$ which is free on $Gw$, and so this in turn
is isomorphic to $\cb(Gw,Tz)$. We may now use the isomorphisms
$\cb^{(T,G)}(GTw,GTz)\cong \cb(Gw,Tz)$ to construct an isomorphic category
$\ct'$ with hom-sets given by the $\cb(Gw,Tz)$.  

Write $X'$ and $Y'$ for the objects of $\ct'$ corresponding to $GTx$ and 
$GTy$. 
Since $Gx\cong c .Mu^*$, $Tx\cong m.Mb $,
$Gy\cong u.u^*.m.c  M$, and $Ty\cong b .m$, the morphisms of 
$\ct'$ may be described as follows: 
\begin{itemize}
\item a morphism $X'\to X'$ is a 2-cell $c .Mu^*\to m.Mb$; 
\item a morphism $X'\to Y'$ is a 2-cell $c .Mu^*\to b.m$; 
\item a morphism $Y'\to X'$ is a 2-cell $u.u^*.m.c  M\to m.Mb $; 
\item a morphism $Y'\to Y'$ is a 2-cell $u.u^*.m.c  M\to b .m$. 
\end{itemize}

We now use various adjunctions to obtain a further isomorphic category
$\ct''$, with objects $X''$ and $Y''$ corresponding to $X'$ and $Y'$, in 
which  
\begin{itemize}
\item a morphism $X''\to X''$ is a 2-cell $c  \to m.Mb.Mu$ in \cm 
\item a morphism $X''\to Y''$ is a 2-cell $c  \to b$ in \cm 
\item a morphism $Y''\to X''$ is a 2-cell $u^*.m.c  M\to u^*.m.Mb $ in \cm 
\item a morphism $Y''\to Y''$ is a 2-cell $u^*.m.c  M\to u^*.b .m$ in
  $\cm$.  
\end{itemize}
First of all, because of the adjunction $u\dashv u^*$, 2-cells
$u.u^*.m.c  M\to m.Mb $ are in bijection with
2-cells $u^*.m.c  M\to u^*.m.Mb $, and similarly
2-cells $u.u^*.m.c  M\to b .m$ are in bijection
with 2-cells $u^*.m.c  M\to u^*.b .m$. Next use
the adjunction $Mu\dashv Mu^*$ to see that 2-cells $c .Mu^*\to m.Mb $ are in
bijection with 2-cells $c \to m.Mb .Mu=i\bullet (j\circ b)$, and
similarly that 2-cells $c.Mu^*\to b .m$ are in bijection with 2-cells $c \to
b.m.Mu$, and finally use a unitality isomorphism $m.Mu\cong 1$ for $M$.  

The desired category $\ct$ has morphisms $X\to X$ and $X\to Y$ as in
$\ct''$. The morphisms in $\ct''$ with domain $Y''$ are morphisms in the
hom-category $\cm(M^2,I)$. In light of the duality $M\dashv M$, the
hom-category $\cm(M^2,I)$ is equivalent to $\cm(M,M)$. 
A morphism $\tau\colon Y''\to X''$ determines a 2-cell 
$$\xymatrix{
&& M \ar[dr]_-{m^*} \ar@{=}[drr] \\
M \ar[r]_-{Mu} \ar@{=}[urr] & M^2 \ar[r]_-{Mm^*} \ar[ur]_-{m} & 
M^3 \ar[r]_-{mM} & 
M^2 \ar[r]_-{u^*M} \ar@{=>}[d]^-{\tau M} & M \\
M \ar[r]_-{Mu} \ar[u]^-{c } & 
M^2 \ar[r]_-{Mm^*} & 
M^3 \ar[u]^-{c  M^2} \ar[r]_-{Mb  M} & 
M^3 \ar[r]_-{mM} & M^2 \ar[u]_-{u^*M} 
}$$
from $c $ to $b ^-$ and this process is
bijective. Similarly a morphism $\sigma\colon Y''\to Y''$ determines a 2-cell  
$$\xymatrix{
&& M \ar[dr]_-{m^*} \ar@{=}[drr] \\
M \ar[r]_-{Mu} \ar@{=}[urr] & M^2 \ar[r]_-{Mm^*} \ar[ur]_-{m} & 
M^3 \ar[r]_-{mM} & 
M^2 \ar[r]_-{u^*M} \ar@{=>}[d]^-{\sigma M} & M \\
M \ar[r]^-{Mu} \ar[u]^-{c } \ar@{=}[drr] & 
M^2 \ar[r]^-{Mm^*} \ar[dr]^-{m} & 
M^3 \ar[u]^-{c  M^2} \ar[r]^-{mM} & 
M^2 \ar[r]^-{b  M} & M^2 \ar[u]_-{u^*M} \\
&& M \ar[ur]^-{m^*} 
}$$
from $c $ to $u^*M.b  M.m^*= (b \circ j)\bullet i$, and this
process is once again bijective.

The bijections described above yield the morphism map of the stated
category isomorphism; the resulting compositions (of the particular morphisms
$X\to Y$ and $Y\to X$) and the identity morphisms in $\ct$ come out as in the
claim. 
\endproof

The morphism in $\ct=\ct^c_b$ corresponding to a morphism $f$ in
$\cb^{(T,G)}$ will be called the {\em transform} of $f$.  

\begin{remark}\label{rem:symmetry_of_T}
In Proposition \ref{prop:transform}, we associated a category $\mathcal T^c_b$
to a naturally Frobenius map-monoidale $M$ in a monoidal bicategory $\mathcal
M$ equipped with a convolution monoid $b$ and a composition comonoid $c$ in
the induced duoidal category $\mathcal M(M,M)$. Replacing the naturally
Frobenius map-monoidale $M=(M,m,u)$ in $\mathcal M$ with $M^*=(M,m^*,u^*)$ in
$\mathcal M^{\oprev}$, we may regard $b$ as a convolution monoid, and we may
regard $c$ as a composition comonoid in the duoidal category $\mathcal
M(M^*,M^*)$. Hence there is an associated category as in  Proposition
\ref{prop:transform}; which is in turn the opposite of the category $\mathcal
T^c_b$.
\end{remark}

The category $\ct^c_b$ in Proposition \ref{prop:transform} depended on a
comonoid $c$ and a monoid $b$. Next we observe that it is functorial in these
inputs. Since all of the constructions involved in the definition of the
transform category are clearly natural in $c$ and $b$, we deduce:

\begin{proposition}\label{prop:T-functoriality}
If $f\colon(c,\delta,\epsilon)\to(c',\delta',\epsilon')$ is a morphism of
comonoids, then there is an induced functor $\ct^{c'}_b\to \ct^c_b$ which
fixes the objects $X$ and $Y$, and which acts on morphisms by composition with
$f$. Similarly if $g\colon(b,\mu,\eta)\to(b',\mu',\eta')$ is a morphism of
monoids, then there is an induced functor $\ct^c_b\to \ct^c_{b'}$ which fixes
the objects; it acts on morphisms by composition with $i\bullet (j\circ g)$,
with $g$, with $g^-$, or with $(g\circ j)\bullet i$, as the case may be.  
\end{proposition}

As well as the functoriality condition given in Proposition
\ref{prop:T-functoriality}, we shall also need to look at transforms
involving tensored monoids or comonoids.

\begin{proposition}\label{prop:T-bullet}
Let $c$ and $d$ be composition-comonoids, and let $b$ be a convolution-monoid
in the duoidal category $\cm(M,M)$ associated to a naturally Frobenius map-
monoidale $M$ in a monoidal bicategory $\cm$. Suppose that $\sigma\colon c\to
b$ and $\sigma'\colon c\to b^-$ define an inverse pair  $X\cong Y$ in
$\ct^c_b$, and similarly that $\tau\colon d\to b$ and $\tau'\colon d\to b^-$
define an inverse pair in $\ct^d_b$. Then the composites  
$$\xymatrix{
c\bullet d \ar[r]^-{\sigma\bullet\tau} & b\bullet b\ar[r]^-{\mu} & b &
c\bullet d \ar[r]^-{\sigma'\bullet\tau'} & b^-\bullet b^- 
\ar[r]^-{\Upsilon_{b,b}} & 
(b\bullet b)^- \ar[r]^-{\mu^-} & b^- }$$
define an inverse pair in $\ct^{c\bullet d}_b$. 
\end{proposition}

\proof
Consider the upper diagram in Figure~\ref{fig}. The upper path gives one
composite in $\ct^{c\bullet d}_b$ of the two displayed morphisms. All of the
small quadrilaterals commute by functoriality of $\bullet$ or $\circ$, or by
naturality of $\psi$ or $\xi$. The large central region commutes by
Lemma~\ref{lem:phi-psi-identities}. The large region on the left commutes by
the fact that $\tau$ and $\tau'$ are mutually inverse in $\ct^d_b$. The
square-shaped pentagon on the right commutes by associativity of $\mu$. The
two triangular regions commute because $\eta$ is a unit for $\mu$, and
counitality of $\xi^0$. Now use the fact that $\sigma$ is inverse to $\sigma'$
in $\ct^c_b$ to show that the lower path is equal to an identity in
$\ct^{c\bullet d}_b$.

This gives one of the inverse laws; the other follows by  the symmetry
described in Remark~\ref{rem:symmetry_of_T}. 
\endproof

\begin{proposition}\label{prop:T-circ}
Let $c$ be a composition-comonoid, and let $a$ and $b$ be convolution-monoids
in the duoidal category $\cm(M,M)$ associated to a naturally Frobenius map-
monoidale $M$ in a monoidal bicategory $\cm$. Suppose that $\sigma\colon c\to
a$ and $\sigma'\colon c\to a^-$ define an inverse pair in $\ct^c_a$, and
similarly that $\tau\colon c\to b$ and $\tau'\colon c\to b^-$ define an
inverse pair in $\ct^c_b$. Then the composites 
$$
\xymatrix{
c \ar[r]^-{\delta} & c\circ c \ar[r]^-{\sigma\circ\tau} & a\circ b &
c \ar[r]^-{\delta} & c\circ c \ar[r]^-{\tau'\circ\sigma'} & b^-\circ a^-
\ar[r]^-{\Xi} & (a\circ b)^- }
$$ 
define an inverse pair in $\ct^c_{a\circ b}$.   
\end{proposition}

\proof
Consider the lower diagram in Figure~\ref{fig}. The upper path gives one
composite in $\ct^c_{a\circ b}$ of the two displayed morphisms. All of the
small quadrilaterals commute by functoriality of $\bullet$ or $\circ$, or by
naturality of $\vartheta$. The large central region commutes by the fact that
$\tau$ is inverse to $\tau'$ in $\ct^c_b$. The large upper region, the
lower region with the curved arrow, and the quadrilateral region just above
it commute by Lemma~\ref{lem:theta}. The irregular region on the left
commutes by coassociativity of $\delta$, and the triangular region next
to it by counitality of $\delta$. Now use the fact that $\sigma$ is inverse to
$\sigma'$ in $\ct^c_a$ to show that the lower path is equal to an identity in
$\ct^c_{a\circ b}$.

This gives one of the inverse laws; the other follows by the symmetry
described in Remark~\ref{rem:symmetry_of_T}. 
\endproof

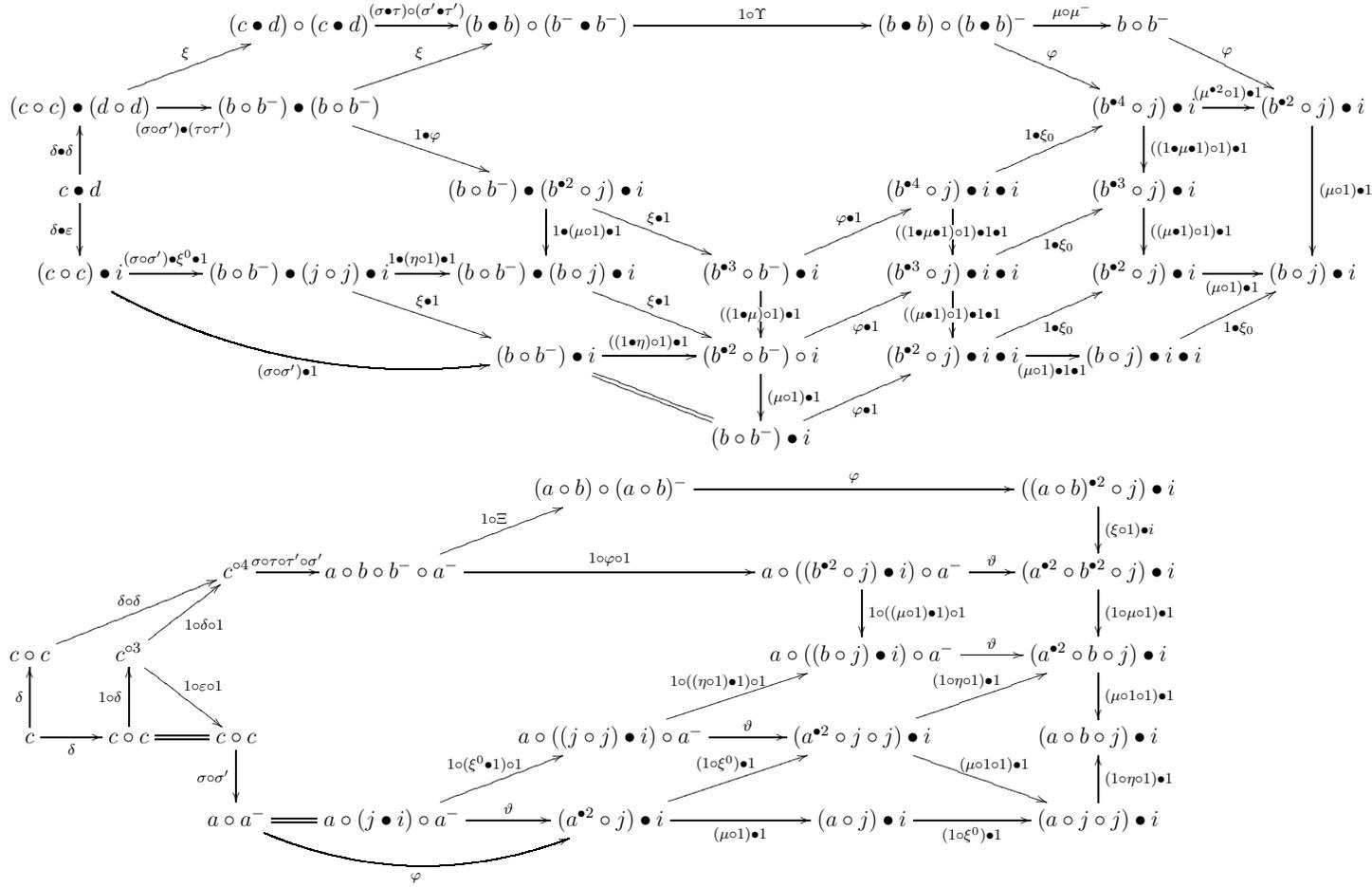
\begin{figure}
  \centering
  \begin{sideways}
\scalebox{0.75}{
\xymatrix{
& (c\bullet d)\circ (c\bullet d) 
\ar[r]^-{(\sigma\bullet\tau)\circ(\sigma'\bullet\tau')} 
& (b\bullet b)\circ(b^-\bullet b^-) \ar[rr]^-{1\circ\Upsilon} 
&& (b\bullet b)\circ(b\bullet b)^- \ar[r]^-{\mu\circ\mu^-} 
\ar[dr]^-{\phi} 
& b\circ b^- \ar[dr]^-{\phi} \\
(c\circ c)\bullet(d\circ d) \ar[ur]^-{\xi} 
\ar[r]_-{\raisebox{-8pt}{${}_{(\sigma\circ\sigma')\bullet(\tau\circ\tau')}$}} 
& (b\circ b^-)\bullet(b\circ b^-) \ar[ur]^-{\xi} 
\ar[dr]^-{1\bullet\phi} &&&&
(b^{\bullet 4}\circ j)\bullet i 
\ar[r]^-{(\mu^{\bullet 2}\circ 1)\bullet1} 
\ar[d]^-{((1\bullet\mu\bullet1)\circ1)\bullet1} 
& (b^{\bullet 2}\circ j)\bullet i \ar[dd]^-{(\mu\circ1)\bullet1} \\
c\bullet d \ar[u]^-{\delta\bullet\delta} \ar[d]_-{\delta\bullet\epsilon}
&&(b\circ b^-)\bullet(b^{\bullet 2}\circ j)\bullet i  
\ar[d]^-{1\bullet(\mu\circ 1)\bullet1} \ar[dr]^-{\xi\bullet1} 
&& (b^{\bullet 4}\circ j)\bullet i\bullet i  
\ar[d]|{((1\bullet\mu\bullet1) \circ 1)\bullet1\bullet1} 
\ar[ur]^-{1\bullet\xi_0} & (b^{\bullet 3}\circ j)\bullet i  
\ar[d]^-{((\mu\bullet1)\circ1)\bullet1}  \\
(c\circ c)\bullet i 
\ar[r]^-{(\sigma\circ\sigma')\bullet\xi^0\bullet 1} 
\ar@/_2pc/[drr]_-{(\sigma\circ\sigma')\bullet1} & 
(b\circ b^-)\bullet(j\circ j)\bullet i 
\ar[dr]^-{\xi\bullet1} \ar[r]^-{1\bullet(\eta\circ1)\bullet1} & 
(b\circ b^-)\bullet(b\circ j)\bullet i  \ar[dr]^-{\xi\bullet1} & 
(b^{\bullet 3}\circ b^-)\bullet i 
\ar[d]|-{((1\bullet\mu)\circ1)\bullet1} \ar[ur]^-{\phi\bullet1} & 
(b^{\bullet 3}\circ j)\bullet i\bullet i 
\ar[d]|{((\mu\bullet1)\circ1)\bullet1\bullet1} 
\ar[ur]_-{1\bullet\xi_0} & 
(b^{\bullet 2}\circ j)\bullet i 
\ar[r]_-{(\mu\circ1)\bullet 1} &
(b\circ j)\bullet i \\
&& (b\circ b^-)\bullet i 
\ar[r]^-{((1\bullet\eta)\circ1)\bullet1} \ar@{=}[dr] 
& (b^{\bullet 2}\circ b^-)\circ i 
\ar[ur]_-{\phi\bullet1} \ar[d]^-{(\mu\circ1)\bullet 1} 
& (b^{\bullet 2}\circ j)\bullet i\bullet i 
\ar[ur]_-{1\bullet\xi_0} \ar[r]_-{(\mu\circ1)\bullet1\bullet1} & 
(b\circ j)\bullet i\bullet i \ar[ur]_-{1\bullet\xi_0} \\
&&& (b\circ b^-)\bullet i \ar[ur]_-{\phi\bullet1} 
}   }
\end{sideways}
\begin{sideways}
\scalebox{0.75}{
\xymatrix{
&&&& (a\circ b)\circ(a\circ b)^- \ar[rr]^-{\phi} && 
((a\circ b)^{\bullet 2}\circ j)\bullet i \ar[d]^-{(\xi\circ1)\bullet i} \\
&& c^{\circ 4} \ar[r]^-{\sigma\circ\tau\circ\tau'\circ\sigma'} 
& a\circ b\circ b^-\circ a^- \ar[ur]^-{1\circ\Xi} 
\ar[rr]^-{1\circ\phi\circ1} 
&& a\circ((b^{\bullet 2}\circ j)\bullet i)\circ a^-  \ar[r]^-{\vartheta} 
\ar[d]^-{1\circ((\mu\circ1)\bullet1)\circ1} & 
(a^{\bullet 2}\circ b^{\bullet 2}\circ j)\bullet i 
\ar[d]^-{(1\circ\mu\circ 1)\bullet 1} \\
c\circ c \ar[urr]^-{\delta\circ\delta} 
& c^{\circ 3} \ar[ur]_-{1\circ\delta\circ1} \ar[dr]^-{1\circ\epsilon\circ 1} 
&&&& a\circ((b\circ j)\bullet i)\circ a^- \ar[r]^-{\vartheta} 
& (a^{\bullet 2}\circ b\circ j)\bullet i 
\ar[d]^-{(\mu\circ1\circ1)\bullet 1} \\
c \ar[r]_-{\delta} \ar[u]^-{\delta} & c\circ c \ar[u]^-{1\circ\delta} 
\ar@{=}[r] 
& c\circ c \ar[d]_-{\sigma\circ\sigma'} 
&& a\circ((j\circ j)\bullet i)\circ a^- 
\ar[ur]^-{1\circ((\eta\circ1)\bullet 1)\circ1} \ar[r]^-{\vartheta} & 
(a^{\bullet 2}\circ j\circ j)\bullet i
\ar[ur]^-{(1\circ\eta\circ1) \bullet 1} 
\ar[dr]^-{(\mu\circ1\circ1)\bullet1} & 
(a\circ b\circ j)\bullet i \\
&& a\circ a^- \ar@{=}[r] \ar@/_2pc/[rr]_-{\phi} & 
a\circ (j\bullet i)\circ a^- 
\ar[ur]^-{1\circ(\xi^0\bullet1)\circ1} \ar[r]^-{\vartheta} & 
(a^{\bullet 2}\circ j)\bullet i \ar[ur]^-{(1\circ\xi^0)\bullet1} 
\ar[r]_-{(\mu\circ1)\bullet1} & (a\circ j)\bullet i 
\ar[r]_-{(1\circ\xi^0)\bullet 1} & 
(a\circ j\circ j)\bullet i \ar[u]_-{(1\circ\eta\circ1)\bullet1} 
}   }
\end{sideways}
\caption{Diagrams for Propositions~\ref{prop:T-bullet}
   and~\ref{prop:T-circ}}\label{fig} 
\end{figure}

Let us take now a monoidal comonad $a$ on  a naturally Frobenius
map-monoidale $M$; it provides us with a
convolution-monoid $(a,\mu,\eta)$ and a composition-comonoid
$(a,\delta,\varepsilon)$ in $\cm(M,M)$. 

\begin{proposition} \label{prop:beta-hat_mixed_morphism}
The Hopf map $\hat{\beta}$ in \eqref{eq:beta-hat} is a morphism $GTx\to
GTy$ in the category $\cb^{(T,G)}$ of Construction \ref{sec:BTG}. 
\end{proposition}

\proof
The algebra and coalgebra structures on $GTx$ are given by the 2-cells
$$\xymatrix{
M^3 \ar[r]^-{aaa}_{~}="1" \ar[d]_-{Mm} & M^3 \ar[r]^-{mM}_{~}="3" \ar[d]^-{Mm} &
M^2 \ar[d]^-{m} && 
M^2 \ar[r]^-{aa}_{~}="5" \ar@{=}[d] & M^2 \ar[r]^-{m} & M \\
M^2 \ar[r]_-{aa}^{~}="2" & M^2 \ar[r]_-{m}^{~}="4" & M && 
M^2 \ar[r]_-{aM}^{~}="6" & M^2 \ar[u]_-{aa} 
\ar@{=>}"1";"2"^{aa_2} 
\ar@{=>}"3";"4"^{\alpha} 
\ar@{=>}"5";"6"^{\delta a} 
}$$
while the algebra and coalgebra structures on $GTy$ are given by the 2-cells
$$\xymatrix{
M^3 \ar[r]^-{aMM} \ar[d]_-{Mm} & 
M^3 \ar[r]^-{mM}_{~}="1" \ar[d]_-{Mm} & M^2 \ar[r]^-{aa}_{~}="3" \ar[d]^-{m} &
M^2 \ar[d]^-{m} &
M^2 \ar[r]^-{aM}_{~}="5" \ar@{=}[d] & M^2 \ar[r]^-{m} & M \ar[r]^-{a} & M \\
M^2 \ar[r]_-{aM} & 
M^2 \ar[r]_-{m}^{~}="2" & M \ar[r]_-{a}^{~}="4" & M &
M^2 \ar[r]_-{aM}^{~}="6" & M^2 \ar[u]_-{aM} 
\ar@{=>}"3";"4"^{a_2} 
\ar@{=>}"1";"2"^{\alpha} 
\ar@{=>}"5";"6"^{\delta M} 
}$$
Compatibility of $\hat{\beta}$ with the algebra structures follows by
associativity of $a_2$; compatibility with the coalgebra structures follows by
coassociativity of $\delta$.  
\endproof

\section{Results}

We continue to suppose that $M=(M,m,u)$ is a map-monoidale in the monoidal
bicategory \cm and that $a$ is a monoidal comonad on $M$. For many results we
shall also need to suppose that $M$ is naturally Frobenius.

\subsection{Antipodes}\label{sect:antipodes}

First we establish the relevant notion of antipode. Recall from
Proposition~\ref{prop:beta-hat_mixed_morphism} that -- using the
notation from Construction \ref{sec:BTG} --  the Hopf morphism
$\hat{\beta}$ is a morphism in $\cb^{(T,G)}$ from $GTx$ to $GTy$. 

\begin{proposition}\label{prop:beta-transform-identity}
For a monoidal comonad $a$ on a naturally Frobenius map-monoidale $M$,
the transform of the Hopf morphism $\hat{\beta}\colon GTx\to GTy$ in
$\cb^{(T,G)}$ of Proposition \ref{prop:beta-hat_mixed_morphism} is the
identity 2-cell $a\to a$.
\end{proposition}

\proof
The morphism in $\ct'$  corresponding to $\hat{\beta}$ will be the morphism
$Gx\to Ty$ given by composing with the unit of $T$ and the counit of
$G$. Composing $\hat{\beta}$ with the counit $GTy\to Ty$ gives $a_2\colon
m.aa\to a.m$ by counitality of $\delta$. Composing this with the unit $Gx\to
GTx$ gives the diagram below on the left 
$$
\xymatrix{
M^2\ar@{=}[rr]_{~}="1"\ar@{<=}[rd];"1"_{\eta_m}\ar[rd]_-m&&
M^2 \ar[rd]^-{Mu^*}\\
& M \ar[ru]_-{m^*}\ar@{=}[rr]&&
M\ar[r]^-a &
M}\quad
\xymatrix{
& M \ar@{=}[rr] \ar[dr]^-{Mu} & & M \ar[r]^-{a} & M \\
M^2 \ar[ur]^-{Mu^*} \ar@{=}[rr]^{~}="1" \ar@{=>}[ur];"1"^{M\epsilon_u} && M^2
\ar[ur]_-{m} }
$$
which, by the unitality of the monoidale $M$, is equal to the diagram on the right. 
Finally, to obtain from this the transform of $\hat \beta$ in $\ct$,
paste with the unit $M\eta_u\colon 1\to Mu^*.Mu$ and use one of the
triangle equations, together with the unitality of the monoidale $M$
once again, to obtain the identity $a\to a$ as the transform of
$\hat{\beta}$. \endproof  

The isomorphism of Proposition~\ref{prop:transform} now allows a description,
in terms of the transformed morphisms, of when the Hopf map
$\hat{\beta}$ in \eqref{eq:beta-hat}  is invertible. 

\begin{theorem}\label{thm:antipode}
For a monoidal comonad $a$ on a naturally Frobenius map-monoidale $M$,
the Hopf morphism $\hat{\beta}$ in \eqref{eq:beta-hat} is
invertible if and only if there exists a 2-cell $\ant\colon a\to a^-$ making
the following diagrams commute. 
$$\xymatrix{
a \ar[r]^-{\delta} \ar[d]_-{\epsilon} & a\circ a \ar[r]^-{a\circ \ant} & 
a\circ a^-\ar[r]^-{\phi_{a,a}} & 
((a\bullet a)\circ j)\bullet i \ar[d]^-{(\mu\circ j)\bullet i} \\
i \ar@{=}[r] & j \bullet i \ar[r]^-{\xi^0\bullet i } & 
(j\circ j)\bullet i \ar[r]^-{(\eta\circ j)\bullet i} & (a\circ j)\bullet i \\
a \ar[r]^-{\delta} \ar[d]_-{\epsilon} & a\circ a \ar[r]^-{\ant\circ a} & 
a^-\circ a \ar[r]^-{\psi_{a,a}}  & 
i\bullet(j\circ(a\bullet a)) \ar[d]^-{i\bullet(j\circ\mu)} \\
i \ar@{=}[r] & i\bullet j \ar[r]^-{i\bullet\xi^0} & 
i\bullet(j\circ j) \ar[r]^-{i\bullet(j\circ\eta)} & i\bullet(j\circ a)
}$$
\end{theorem}

\proof
We have seen in Proposition~\ref{prop:beta-transform-identity} that the
transform of the morphism $\hat\beta\colon GTx\to GTy$ in Proposition
\ref{prop:beta-hat_mixed_morphism} is the identity $1\colon a\to a$.  
The transform of a morphism $\hat\beta'\colon GTy\to GTx$ will be a
2-cell $\ant\colon a\to a^-$. The two conditions in the theorem are transforms
to $\ct$ of the two equations for $\hat\beta'$ to be inverse to
$\hat\beta$ in the category $\cb^{(T,G)}$.
\endproof

We call a 2-cell $\ant\colon a\to a^-$ satisfying the conditions in
Theorem~\ref{thm:antipode}  an {\em antipode} for the monoidal comonad 
$a$.

\begin{example}\label{ex:trivial-antipode}
Consider the $\circ$-trivial bialgebra $i$ of Example~\ref{ex:trivial} 
in a duoidal category $\cm (M,M)$ induced by a naturally Frobenius
map-monoidale $(M,m,u)$. The Hopf map $\hat\beta$ is in fact the
identity 2-cell $1\colon m\to m$, seen as a morphism $GTx\to GTy$; this is of
course invertible. By Theorem~\ref{thm:antipode}, therefore, there is an
antipode $i\to i^-$. This can be calculated by transforming $1\colon m\to m$,
now seen as a morphism $GTy\to GTx$. An explicit calculation shows that this
gives $\Xi_0$.

Now consider the $\bullet$-trivial bialgebra $j$ of
Example~\ref{ex:trivial}. The Hopf map $\hat\beta$ has the form  
$$\xymatrix{
MM \ar[r]^-{u^*u^*} \ar[dd]_-{u^*1} & 
I \ar[r]^-{uu} 
\ar@{=}
[rrrrrdd]&
MM \ar[rrrr]^-{m}  &&&& 
M \\
&& MM \ar[lu]^-{u^*u^*}
\\
M \ar[r]_{u1}^(.6){~}="2" \ar[ruu]^-{u^*}_{~}="1"& 
MM \ar[r]_{m}^(.6){~}="4"  \ar@{=}[ru]_-{~}="3"
&
M \ar[rrrr]_-{u^*} \ar[u]_-{m^*}&&&&
I \ar[uu]_-{u} 
\ar@{=>}"1";"2"^{\eta_u u^*}
\ar@{=>}"3";"4"^{\eta_m}
}$$
and by the unitality of the monoidale $M$ this is equal to the canonical
isomorphism  
$$\xymatrix{
MM \ar[r]^-{u^*u^*} \ar[d]_-{u^*1} & 
I \ar[r]^-{uu}_{~}="1" \ar@{=}[rrrrrd]&
MM \ar[rrrr]^-{m}  &&&&
M \\
M \ar[r]_-{u1}\ar@{=}@/^1.5pc/[rr] & 
MM \ar[r]_-{m}^{~}="2" &
M \ar[rrrr]_-{u^*} &&&& 
I \ar[u]_-{u} 
}$$ 
involving two copies of the unit isomorphism of  the monoidale $M$. Thus
by Theorem~\ref{thm:antipode} once again, there is an antipode $j\to j^-$, 
given by transforming the inverse. An explicit calculation shows that this is
$\Upsilon_0$.
\end{example}

We observed in Section \ref{sec:more-on-NatFrob}  that the meaning of $a^-$ is
unchanged whether we regard $a$ as a monoidal comonad on $(M,m,u)$ or a
monoidal comonad on $(M,m,u)^*$, but that the roles of $\phi_{a,a}$ and
$\psi_{a,a}$ are interchanged. Given this, it is straightforward to see that
moving from $(M,m,u)$ to $(M,m,u)^*$ interchanges the roles of the two
equations for an antipode.

We deduce:

\begin{theorem}\label{thm:Hopf}
Let $(M,m,u)$ be a naturally Frobenius map-monoidale in a monoidal bicategory
\cm, and let $a$ be a monoidal comonad on $(M,m,u)$. The following conditions
are equivalent: 
\begin{enumerate}[(a)]
\item the Hopf map $\hat{\beta}$ of \eqref{eq:beta-hat} is invertible; 
\item the co-Hopf map $\hat{\zeta}$ of \eqref{eq:zeta-hat} is
invertible;  
\item there exists an antipode $\ant\colon a\to a^-$. 
\end{enumerate}
\end{theorem}

The well-known fact, that the antipode of a Hopf algebra is an algebra and
coalgebra anti-homomorphism, takes the following form in our setting. 

\begin{theorem}\label{thm:antipode_(co)monoid_map}
Let $(M,m,u)$ be a naturally Frobenius map-monoidale in a monoidal bicategory
\cm, and let $a$ be a monoidal comonad on $(M,m,u)$ obeying the equivalent
conditions in Theorem \ref{thm:Hopf}. Then the antipode $\ant$ is a monoid
morphism $(a,\mu,\eta)\to (a^-,\mu^-,\eta^-)$ and a comonoid morphism
$(a,\delta,\varepsilon)\to (a^-,\delta^-,\varepsilon^-)$. 
\end{theorem}
\proof
By Proposition~\ref{prop:transform}, we obtain a category $\ct^j_a$, where $j$
is the  composition-comonoid $(j,\xi^0,\xi_0^0)$. In order to see that $\ant$
preserves the unit, we claim that both composites 
\begin{equation}\label{eq:antipode-unit}
\xymatrix{
j\ar[r]^-\eta &
a\ar[r]^-{\ant}&
a^-}
\quad \textrm{and} \quad
\xymatrix{
j\ar[r]^-{\Upsilon_0} &
j^-\ar[r]^-{\eta^-}&
a^-}
\end{equation}
yield the inverse to $\eta$ in $\ct^j_a$. Note first that $\eta\colon j\to a$
is a comonoid morphism, and so by Proposition~\ref{prop:T-functoriality}
induces a functor $\ct^a_a\to \ct^j_a$ sending $1_a\colon X\to Y$ to
$\eta\colon X\to Y$ and $\ant\colon Y\to X$ to the composite $\ant.\eta\colon
Y\to X$. Since $1_a$ is inverse to $\ant$ in $\ct^a_a$, and functors preserve
inverses, it follows that the first expression in \eqref{eq:antipode-unit} is
the inverse to $\eta$ in $\ct^j_a$. On the other hand, $\eta\colon j\to a$ is
a morphism of monoids, and so by Proposition~\ref{prop:T-functoriality}
induces a functor $\ct^j_j\to \ct^j_a$ sending $1_j$ to $\eta$ and sending
$\Upsilon_0$ to $\eta^-.\Upsilon_0$. Recall from
Example~\ref{ex:trivial-antipode}  that $\Upsilon_0\colon j\to j^-$ is an
antipode for the bimonoid $j$,  and so is inverse in $\ct^j_j$ to
$1_j$. Functors preserve inverses, and so also the second expression in
\eqref{eq:antipode-unit} is the inverse to $\eta$ in $\ct^j_a$. This proves
that $\ant$ is compatible with the units. 

The case of counits is similar: we prove that the composites
\begin{equation}\label{eq:antipode-counit}
\xymatrix{
a\ar[r]^-\varepsilon &
i\ar[r]^-{\Xi_0} &
i^-}
\quad \textrm{and} \quad
\xymatrix{
a\ar[r]^-\ant &
a^-\ar[r]^-{\varepsilon^-}&
i^-}
\end{equation}
are both inverse in $\ct^a_i$ to $\epsilon$, and so are equal, using the fact 
that $\Xi_0\colon i\to i^-$ is an antipode. 

Since $\bullet$ is opmonoidal with respect to $\circ$, the $\bullet$-product
of two $\circ$-comonoids is a \hbox{$\circ$-comonoid}; in particular 
$a\bullet a$ is a comonoid with comultiplication $\xi.(\delta\bullet\delta)$ 
and counit $\xi_0.(\epsilon\bullet\epsilon)$. Furthermore, $\mu\colon a\bullet 
a\to a$ is a comonoid morphism, and so induces a functor 
$\ct^a_a\to \ct^{a\bullet a}_a$, sending $1_a\colon X\to Y$ to $\mu$ and sending 
$\ant\colon Y\to X$ to the composite $\ant.\mu$ appearing on the left of
\begin{equation}\label{eq:antipode-multiplication}
\xymatrix@C=15pt{
a\bullet a\ar[r]^-\mu &
a\ar[r]^-{\ant}&
a^-}
\ \  \textrm{and} \ \ 
\xymatrix@C=15pt{
a\bullet a\ar[r]^-{\ant \bullet \ant} &
a^- \bullet a^- \ar[r]^-\Upsilon &
(a\bullet a)^- \ar[r]^-{\mu^-}&
a^-}
\end{equation}
which is therefore inverse to $\mu$ in $\ct^{a\bullet a}_a$; on the other
hand, the second  expression in \eqref{eq:antipode-multiplication}  is inverse
to $\mu$ by Proposition~\ref{prop:T-bullet}, thus the two composites are
equal. This proves compatibility with the multiplication. 
 
Finally we turn to compatibility with the comultiplication. This time we use
the monoid $(a\circ a,(\mu\circ\mu).\xi,(\eta\circ\eta).\xi^0)$, the monoid
homomorphism $\delta\colon a\to a\circ a$,  and the induced functor
$\ct^a_a\to\ct^a_{a\circ a}$. This sends the inverse $\ant$ of $1_a$ to an
inverse $\delta^-.\ant$ of $\delta$ as on the left of 
\begin{equation}\label{eq:antipode-comultiplication}
\xymatrix@C=15pt{
a\ar[r]^-\ant &
a^-\ar[r]^-{\delta^-} &
(a\circ a)^-}
\ \  \textrm{and} \ \ 
\xymatrix@C=15pt{
a\ar[r]^-\delta &
a\circ a\ar[r]^-{\ant \circ \ant}&
a^- \circ a^-\ar[r]^-{\Xi} &
(a\circ a)^- .}
\end{equation}
On the other hand, the second expression in
\eqref{eq:antipode-comultiplication}  is inverse to $\delta$ by
Proposition~\ref{prop:T-circ}, thus the two composites are equal. This proves
compatibility with the comultiplication.   
\endproof

\subsection{The Galois maps}\label{sect:Galois}

In this section we investigate the relationship between the invertibility of
the Hopf maps of Section~\ref{sec:HopfMap} and the invertibility
of the Galois maps of Sections~\ref{sec:Galois} and~\ref{sec:coGalois}. 
This seems to require another assumption on the
map-monoidale. Recall that a functor is said to be {\em conservative} when it
reflects isomorphisms.

We say that an object $M$ of the monoidal bicategory \cm is {\em well-pointed}
if there is a morphism $v\colon I\to M$ in \cm for which the functor 
$\cm(v1,M)\colon\cm(M^2,M)\to\cm(M,M)$ induced by composition with $v1\colon
M\to M^2$ is conservative.  

\begin{example}
In the situation of Example~\ref{ex:braided_bialg}, where \cm has a single
object and $M$ is the trivial map-monoidale, we may take $v$ to be the
identity (which is the unit object of the corresponding braided monoidal
category). 
\end{example}
 
\begin{example}
In the situation of Example~\ref{ex:bialgebroid}, where $\cm=\Mod$ and $M$ is
a commutative ring, seen as a monoidale in \Mod, we may take $v=u$. For then
$\cm(M^2,M)$ is the category of left $M^2$-, right $M$-modules, and $\cm(M,M)$
is the category of left $M$-, right $M$-modules, while $\cm(v1,M)$ is given by
restriction of scalars. Similarly, in the situation of
Example~\ref{ex:category}, where $\cm=\Span\co$ and $M$ is just a set,  we may 
take $v=u$: this is the unique map from  $M$ to the singleton $1$, seen as a
span from $1$ to $M$. 
\end{example}

\begin{example}
In the situation of Example~\ref{ex:weak-bialgebra}, where $\cm=\Mod$ and $M$
is $R\op R$ for a separable Frobenius ring $R$, the unit $n\colon I\to R\op R$
does not have the required property; instead, we take the unique homomorphism
of rings $I\to R\op R$ as our $v$, so that $\cm(v1,M)$ is once again given by
restriction of scalars.  
\end{example}

\begin{example}
In the situation of  Example~\ref{ex:Prof}, where $\cm=\Prof$ and $M$ is a
monoidal category with duals, we may take $v$ to be the profunctor given, as a
functor $M\op=1\x M\op\to\Set$, by $v(x)=\sum_{y\in M} M(x,y)$.
\end{example}

\begin{theorem}\label{thm:coGalois}
For a map-monoidale $(M,m,u)$ in a monoidal bicategory \cm, and a monoidal
comonad $a$ on $(M,m,u)$, consider the following assertions.
\begin{enumerate}[(a)]
\item the Hopf map $\hat{\beta}\colon m.aa\to a.m.a1$  in
  \eqref{eq:beta-hat}  is invertible;
\item the co-Galois maps $\zeta_{p,x}\colon p\bullet(x\circ a)\to (p\bullet
  x)\circ a$  in \eqref{eq:zeta}  are invertible for  every 1-cell $x$ and
  every comodule $p$;   
\item the co-Galois maps $\zeta_{p,i}\colon p\bullet a\to(p\bullet i)\circ a$
  are invertible for every comodule $p$.    
\item For any left $a$-module $(q,\gamma)$, right $a$-comodule $(p,\varrho)$,
  and 1-cell $x:M\to M$, the 2-cell 
\begin{equation}\label{eq:left_can}
\xymatrix{
p\bullet  (x\circ q) \ar[r]^-{\varrho\bullet 1} &
(p\circ a)\bullet (x\circ q) \ar[r]^-\xi &
(p\bullet x) \circ (a\bullet q) \ar[r]^-{1\circ\gamma} &
(p\bullet x)\circ q}
\end{equation}
is invertible.
\end{enumerate}
Then (d)$\Rightarrow$(b)$\Rightarrow$(c) and (a)$\Rightarrow$(b). If the
object $M$ is well-pointed, then also (c)$\Rightarrow$(a). If the
map-monoidale $(M,m,u)$ is naturally Frobenius, then (a)$\Rightarrow$(d). In 
particular, if $(M,m,u)$ is a well-pointed naturally Frobenius map-monoidale 
then all four conditions are equivalent.
\end{theorem}

\proof
Substituting $(q,\gamma)=(a,\mu)$ in (d) we obtain (b), and substituting $x=i$
in (b) we get (c).

For (a)$\Rightarrow$(b), first observe that the co-Galois maps $\zeta_{p,x}$
are natural with respect to comodule morphisms $(p,\rho)\to(p',\rho')$, thus
they will be invertible for every comodule $(p,\rho)$ if and only if they are
invertible for every cofree comodule $(y\circ a,y\circ \delta)$. For a
cofree comodule the co-Galois map has the 
form   
$$\xymatrix{
M \ar[r]^-{m^*} & M^2 \ar[r]^-{yx} & 
M^2 \ar@/^2pc/[rr]^{aa}_(0.71){~}="1" \ar[r]_-{aM}  & 
M^2 \ar[r]_{aa}="2"^(0.4){~}="4" \ar[d]_-{m} & MM \ar[d]^-{m} \\
&&& M \ar[r]_{a}^{~}="3" & M 
\ar@{=>}"2";"3"^{a_2} 
\ar@{=>}"1";"4"_(0.3){\delta a}
}$$
and this will clearly be invertible if the Hopf map is invertible. Thus (a)
implies (b). 

Next we show that, under the additional assumption that $M$ is well-pointed,
(c) implies (a). Consider the case $x=i$ and $y=v.u^*$. Using the counit
isomorphism $u^*1.m^*\cong 1$, the co-Galois map $\zeta_{y\circ a,x}$
becomes
$$\xymatrix{
M \ar[r]^-{vM} & 
M^2 \ar@/^2pc/[rr]^{aa}_(0.71){~}="1" \ar[r]_-{aM}  & 
M^2 \ar[r]_{aa}="2"^(0.4){~}="4" \ar[d]_-{m} & MM \ar[d]^-{m} \\
&& M \ar[r]_{a}^{~}="3" & M 
\ar@{=>}"2";"3"^{a_2} 
\ar@{=>}"1";"4"_(0.3){\delta a}
}$$
and invertibility of this, together with the assumption that $\cm(vM,M)$ is
conservative implies invertibility of the Hopf map. 

Assume now that $(M,m,u)$ is a naturally Frobenius map-monoidale. 
Since the 2-cell \eqref{eq:left_can} is natural both in the module $q$ and
the comodule $p$, it is an isomorphism for every $q$ and $p$ if and only if it
is so for the free modules $(a\bullet z,\mu\bullet z)$ and the free comodules
$(y\circ a,y\circ\delta)$ (for arbitrary 1-cells $y,z:M\to M$). With these
choices it takes the form  
$$
\xymatrix{
M\ar[r]^-{m^*} &
M^2\ar[r]^-{yx} &
M^2 \ar[rr]^-{1m^*} \ar[d]_-{a1} &&
M^3 \ar[r]^-{aaz}_-{~}="5" \ar[d]_-{a11} &
M^3 \ar[r]^-{1m} \ar@{=}[d] &
M^2 \ar[dd]^-m\\
&&
M^2 \ar@{=}[r]_(.35){~}="1" \ar[d]_-m &
M^2 \ar[r]^-{1m^*} &
M^3 \ar[r]^-{aaz}="6"_-{~}="3" &
M^3 \ar[d]^-{m1} \\
&&
M \ar[rr]_-{m^*} \ar[ru]_-{m^*}^(.44){~}="2" &&
M^2 \ar[r]_-{az}^-{~}="4" \ar[u]^-{m^*1} &
M^2 \ar[r]_-m &
M
\ar@{=>}"1";"2"_-{\eta_m}
\ar@{=>}"3";"4"^-{\mu 1}
\ar@{=>}"5";"6"^-{\delta 11}}
$$
which is equal to 
$$
\xymatrix{
M\ar[r]^-{m^*} &
M^2\ar[r]^-{yx} &
M^2 \ar[r]^-{1m^*} \ar[d]_-{a1} &
M^3 \ar[r]^-{aaz}_-{~}="5" \ar[d]^-{a11} &
M^3 \ar[r]^-{1m} \ar[dd]^-{m1} &
M^2 \ar[dd]^-m\\
&&
M^2 \ar[r]^-{1m^*}_{~}="1" \ar[d]_-m &
M^3 \ar[d]^-{m1} \\
&&
M \ar[r]_-{m^*}^{~}="2"  &
M^2 \ar[r]_-{az}^-{~}="4" &
M^2 \ar[r]_-m &
M
\ar@{=>}"1";"2"^-{\pi}
\ar@{=>}"5";"4"^-{\hat \beta 1}}
$$
proving (a)$\Rightarrow$(d).
\endproof

Note that the above proof of the implication (a)$\Rightarrow$(d) makes
use, in fact, only of the invertibility of the 2-cell $\pi$ but not the
invertibility of $\pi'$. 

The dual result, relating invertibility of the co-Hopf map of Section~\ref{sec:HopfMap} 
 and the Galois maps  of Section~\ref{sec:Galois}, partly follows by symmetry considerations. 
We say that an object $M$ of a monoidal bicategory \cm is {\em
well-copointed} whenever it is well-pointed as an object of
$\cm\op$. That is, when there is a morphism $w\colon M\to I$ for which 
$\cm(M,w1)\colon\cm(M,M^2)\to\cm(M,M)$, the functor induced by
composition with $w1\colon M^2\to M$, is conservative. 

Our main interest is of course when $M$ underlies a naturally Frobenius 
map-monoidale $(M,m,u)$ in $\cm$, and then the object $M$ is well-pointed if 
and only if it is well-copointed. Indeed, for $v:I\to M$ the induced functor 
$\cm(v1,M):\cm(M^2,M) \to \cm(M,M)$ is conservative if and only if the functor 
$\cm(M,v^+1)\colon\cm(M,M^2)\to\cm(M,M)$ is so, for 
$$
v^+=
\xymatrix{
M \ar[r]^-{v1} &
M^2 \ar[r]^-m &
M \ar[r]^-{u^*} &
I  .}
$$
In particular, in each of Examples~\ref{ex:braided_bialg}, 
\ref{ex:bialgebroid}, \ref{ex:weak-bialgebra}, \ref{ex:category}, and 
\ref{ex:Prof}, the relevant object  $M$ is well-copointed. 

\begin{theorem}\label{thm:Galois}
For a map-monoidale $(M,m,u)$ in a monoidal bicategory \cm, and a monoidal
comonad $a$ on $(M,m,u)$, consider the following assertions. 
\begin{enumerate}[(a)]
\item the co-Hopf map $\hat{\zeta}\colon aa.m^*\to 1a.m^*.a$ of
  \eqref{eq:zeta-hat} is invertible; 
\item the Galois maps $\beta_{q,x}\colon (q\circ x)\bullet a\to
  q\circ(x\bullet a)$ of \eqref{eq:beta} are invertible for every 1-cell $x$
  and every module $q$;
\item the Galois maps $\beta_{q,j}\colon (q\circ j)\bullet a\to q\circ a$ are
  invertible for every module $q$. 
\item For any right $a$-module $(q,\gamma)$, left $a$-comodule $(p,\varrho)$,
  and 1-cell $x:M\to M$, the 2-cell
\begin{equation}\label{eq:right_can}
\xymatrix{
(q\circ x) \bullet  p \ar[r]^-{1 \bullet \varrho} &
(q\circ x)\bullet (a\circ p) \ar[r]^-\xi &
(q\bullet a) \circ (x\bullet p) \ar[r]^-{\gamma\circ1} &
q\circ(x\bullet p)}
\end{equation}
is invertible.
\end{enumerate}
Then (d)$\Rightarrow$(b)$\Rightarrow$(c). If the object $M$ is
well-copointed, then also (c)$\Rightarrow$(a). If the map-monoidale $(M,m,u)$
is naturally Frobenius, then (a)$\Rightarrow$(d). In particular, if $(M,m,u)$ 
is a well-pointed naturally Frobenius map-monoidale then all four conditions 
are equivalent.
\end{theorem}

\proof
Substituting $(p,\varrho)=(a,\delta)$ in (d) we obtain (b) and substituting
$x=j$ in (b) we obtain (c). 

In order to see that, under the additional assumption that $M$ is
well-copointed, (c) implies (a), consider the 1-cell 
$$
y:=\xymatrix{
M \ar[r]^-{u^*} &
I \ar[r]^-u & 
M \ar[r]^-{m^*} &
M^2 \ar[r]^-{w1} &
M}
$$
--- where $\cm(M,wM)$ is assumed to be conservative --- and make in (b) the
choices $(q,\gamma)=(y\bullet a,y\bullet \mu)$ and $x=j$. Then the Galois map
$\beta_{q,x}$ differs only by coherence isomorphisms from 
$$
\xymatrix{
M^2 \ar@/^2pc/[rr]^-{aa}_(0.35){~}="1" \ar[r]_-{aa}="3"^(0.72){~}="2" & 
M^2 \ar[r]_-{1a} & M^2 \ar[r]^-{w} & M \\
M \ar[u]^-{m^*} \ar[r]_-{a}^{~}="4" & M \ar[u]_-{m^*} 
\ar@{=>}"1";"2"_{a\delta}
\ar@{=>}"3";"4"^{a^2} 
}
$$
and invertibility of this, together with the assumption that $\cm(M,w M)$ is
conservative implies invertibility of the co-Hopf map.  

It remains to prove that, whenever $(M,m,u)$ is a naturally Frobenius
map-monoidale, (a) implies (d). Any right $a$-module $(q,\gamma)$ determines a
module  
$$
\xymatrix{
a^-\bullet q^-\ar[r]^-\Upsilon &
(q\bullet a)^- \ar[r]^-{\gamma^-} &
q^-}
$$
for the bimonoid $a^-$ in
\cm\oprev. Symmetrically, any left $a$-comodule $(p,\varrho)$ determines an
$a^-$-comodule 
$$ 
\xymatrix{
p^-\ar[r]^-{\varrho^-} &
(a\circ p)^- \ar[r]^-{\Xi^{-1}} &
p^- \circ a^-.}
$$
With these constructions, strong duoidality of the functor $(-)^-:\cm(M,M)\to$
$\cm\oprev(M,M)$ yields a commutative diagram
$$
\xymatrix{
p^-\bullet (x^- \circ q^-) \ar[d]_-{\eqref{eq:left_can}}
\ar[r]^-{1\bullet \Xi} &
p^- \bullet  (q\circ x)^-  \ar[r]^-\Upsilon &
((q\circ x)\bullet p)^- \ar[d]^-{\eqref{eq:right_can}^-} \\
(p^-\bullet x^-)\circ q^- \ar[r]_-{\Upsilon\circ1} &
 (x\bullet p)^-\circ q^- \ar[r]_-\Xi &
(q\circ(x\bullet p))^- }
$$
whose horizontal arrows are invertible.
Thus in view of Theorem \ref{thm:coGalois} and \eqref{eq:beta-zeta_duality},
we conclude that (a) implies (d). 
\endproof
 
\subsection{The case of trivial bialgebras}

For the duoidal hom-category of a naturally Frobenius map-monoidale, it
follows by Example \ref{ex:trivial-antipode} and Theorem \ref{thm:Hopf}  
that the co-Hopf morphism $\hat\zeta$ is invertible for the $\circ$-trivial
bialgebra $i$; and the Hopf morphism $\hat\beta$ is invertible for the
$\bullet$-trivial bialgebra $j$ of Example~\ref{ex:trivial}. 

In any duoidal category $\cd$ with monoidal structures $(\circ,i)$ and
$(\bullet,j)$, we write $\cd^j$ for the category of right
$j$-comodules, and $V\colon \cd^j \to \cd$ for the forgetful functor; of
course this has a right adjoint $G\colon\cd\to\cd^j$ sending an object $x$ to
$x\circ j$ with its canonical comodule structure.

Similarly, we write $\cd_i$ for the category of right $i$-modules, and $U\colon
\cd_i\to\cd$ for the forgetful functor; of course this has a left adjoint
$F\colon\cd\to\cd_i$ sending an object $x$ to $x\bullet i$ with its canonical
module structure.  

Thus we have a pair of adjunctions
\begin{equation}\label{eq:i-j_adjunction}
\xymatrix{
\cd^j \ar@<1ex>[r]^-{V} \ar@{}[r]|{\bot} & 
\cd \ar@<1ex>[l]^-{G} \ar@<1ex>[r]^-{F} \ar@{}[r]|{\bot} & 
\cd_i  \ar@<1ex>[l]^-{U} }
\end{equation}
and so a composite adjunction $FV\dashv GU$.

\begin{proposition}\label{prop:trivial-cotrivial}
If $(M,m,u)$ is a naturally Frobenius map-monoidale in a monoidal bicategory
\cm, then the composite  adjunction constructed as in \eqref{eq:i-j_adjunction} 
defines an equivalence $\cm(M,M)^j$ $\simeq \cm(M,M)_i$. 
\end{proposition}

\proof
Write $n\colon 1\to UF$ and $e\colon FU\to 1$ for the unit and counit of the
adjunction $F\dashv U$, and write $h\colon 1\to GV$ and $d\colon VG\to 1$ for
the unit and counit of the adjunction $V\dashv G$. Then the composite
adjunction $FV\dashv GU$ has unit and counit given by  
$$\xymatrix @R0pc {
1 \ar[r]^-{h} & GV \ar[r]^-{GnV} & GUFV}
\quad \textrm{and}\quad
\xymatrix @R0pc {
FVGU \ar[r]^-{FdU} & FU \ar[r]^-{e} & 1. }$$
First consider the unit. Since $V$ is conservative, the unit will be
invertible if and only if the composite  
$$\xymatrix @C3pc {
V \ar[r]^-{Vh} & VGV \ar[r]^-{VGnV} & VGUFV }$$
is invertible; in other words, if for each $j$-comodule $(p,\rho)$ the
corresponding component is invertible. But this will be true for every
$j$-comodule if and only if it is true for every cofree comodule $x\circ j$. 

The component at $x\circ j$ of the unit is the composite
$$\xymatrix@C=35pt{
x\circ j \ar[r]^-{1 \circ\xi^0} & x\circ j\circ j \ar@{=}[r] & 
((x\circ j)\bullet j)\circ j \ar[r]^-{(1\bullet\xi^0_0)\circ1} &
((x\circ j)\bullet i)\circ j }$$
which, up to composition with unit isomorphisms for the duoidal category
$\cm(M,M)$, is the co-Galois morphism $\zeta_{x\circ j,i}$ for the monoidal
comonad $j$. This is invertible by Example~\ref{ex:trivial-antipode} and
Theorem~\ref{thm:coGalois}.

As for the counit, since $U$ is conservative, this will be invertible if and
only if the composite  
$$\xymatrix{
UFVGU \ar[r]^-{UFdU} & UFU \ar[r]^-{Ue} & U }$$
is invertible; in other words, if for each $i$-module $(q,\gamma)$ the
corresponding component is invertible. But this will be true for every
$i$-module if and only if it is true for every free $i$-module $x\bullet i$.  

The component at $x\bullet i$ of the counit is the composite 
$$
\xymatrix@C=40pt{
((x\bullet i)\circ j) \bullet i 
\ar[r]^-{(1\circ\xi^0_0)\bullet1} &
((x\bullet i)\circ i) \bullet i \ar@{=}[r] & 
x\bullet i\bullet i \ar[r]^-{1 \bullet \xi_0} & x\bullet i }
$$
which, up to composition with unit isomorphisms for the duoidal category
$\cm(M,M)$, is the Galois morphism $\beta_{x\bullet i,j}$ for the monoidal
comonad $i$. This is invertible by  Example~\ref{ex:trivial-antipode}, 
Theorem \ref{thm:Hopf}, and Theorem~\ref{thm:Galois}.  
\endproof

\begin{proposition}
For a naturally Frobenius map-monoidale $(M,m,u)$ in a monoidal bicategory
\cm, the equivalence in Proposition \ref{prop:trivial-cotrivial} is a strong
monoidal equivalence $\cm(M,M)^j \simeq \cm(M,M)_i$.
\end{proposition}

\begin{proof}
We claim that a strong monoidal structure on the functor $\cm(M,M)^j \to
\cm(M,M)_i$ in Proposition \ref{prop:trivial-cotrivial} is given by the
evident nullary part $j\bullet i =i$ and the binary part obtained from the
2-cell \eqref{eq:left_can} for the $\bullet$-trivial bicomonad $j$,
substituting $x=i$ and $q=p'\bullet i$ for any $j$-comodule $p'$
(with trivial $j$-action $j\bullet p'\bullet i=p'\bullet i$). The resulting
2-cell
$$
\xymatrix{
p\bullet p'\bullet i \ar[r]^-{\rho\bullet1\bullet1} & 
(p\circ j)\bullet p'\bullet i \ar@{=}[r] &
(p\circ j)\bullet(i\circ(p'\bullet i)) \ar[r]^-{\xi} & 
(p\bullet i)\circ(p'\bullet i) }
$$
is clearly natural in both $j$-comodules $p$ and $p'$. It is an isomorphism by
Example~\ref{ex:trivial-antipode} and Theorem~\ref{thm:coGalois}.
It is a morphism of $i$-modules by associativity and naturality of
$\xi$, by functoriality of $\bullet$, and by counitality of $\xi_0$; see the
first diagram of Figure~\ref{fig:monoidal_equivalence}. Unitality of the
monoidal structure holds since both
$$
\xymatrix@C=18pt{
p\bullet i =
j\bullet p\bullet i \ar[r]^-{\xi^0 \bullet 1 \bullet 1} &
(j\circ j) \bullet p\bullet i =
(j\circ j) \bullet (i\circ(p\bullet i)) \ar[r]^-\xi &
(j\bullet i)\circ (j \bullet p\bullet i)  =
p\bullet i}
$$
and
$$
\xymatrix@C=18pt{
p\bullet i \ar[r]^-{\varrho \bullet 1} &
(p\circ j)\bullet i  \ar@/_2pc/[rrr]_-{(1\circ \xi^0_0)\bullet 1}
\ar@{=}[r] &
(p\circ j)\bullet (i \circ i) \ar[r]^-\xi &
(p\bullet i)\circ (j\bullet i) \ar@{=}[r] &
(p\circ i)\bullet i \ar@{=}[r] & p\bullet i}
$$
are equal to the identity 2-cell by the unitality of the monoidal structure
$(\xi,\xi^0)$ and by the counitality of $\varrho$, respectively.
Associativity of the monoidal structure follows by commutativity of the
second diagram of Figure \ref{fig:monoidal_equivalence}.
Its various regions commute by the (co)associativity and (co)unitality
properties of $\xi$ and its naturality, and by the coassociativity of
$\varrho$. 
\begin{figure}
\centering
\begin{sideways}
\scalebox{.75}{ 
\xymatrix{
p\bullet p'\bullet i\bullet i \ar@{=}[r] \ar@{=}[dd] & 
p\bullet(i\circ(p'\bullet i))\bullet i \ar@{=}[d] \ar[r]^{\rho\bullet1\bullet 1} & 
(p\circ j)\bullet(i\circ(p'\bullet i))\bullet i \ar@{=}[d] \ar[r]^{\xi\bullet1} & 
((p\bullet i)\circ(p'\bullet i))\bullet i \ar@{=}[d] \\
& 
p\bullet (i\circ(p'\bullet i))\bullet ( i\circ i) 
\ar[r]^{\rho\bullet1\bullet1} \ar[d]_{1\bullet\xi} & 
(p\circ j)\bullet(i\circ(p'\bullet i))\bullet(i\circ i) \ar[r]^{\xi\bullet 1} 
\ar[d]_{1\bullet\xi} & 
((p\bullet i)\circ(p'\bullet i))\bullet(i\circ i) \ar[d]^{\xi} \\
p\bullet p'\bullet i\bullet i \ar[d]_{1\bullet 1\bullet \xi_0} &  
p\bullet((i\bullet i)\circ(p'\bullet i\bullet i)) 
\ar[l]_-{1\bullet(\xi_0\circ1)} \ar[d]_{1\bullet(\xi_0\circ(1\bullet\xi_0))} & 
(p\circ j)\bullet((i\bullet i)\circ(p'\bullet i\bullet i)) \ar[r]^{\xi} 
\ar[d]_{1\bullet(\xi_0\circ(1\bullet\xi_0))} & 
(p\bullet i\bullet i)\circ(p'\bullet i\bullet i) 
\ar[d]^{(1\bullet\xi_0)\circ(1\bullet\xi_0)} \\
p\bullet p'\bullet i \ar@{=}[r] & p\bullet(i\circ(p'\bullet i)) 
\ar[r]^{\rho\bullet1} & 
(p\circ j)\bullet(i\circ(p'\bullet i)) \ar[r]^{\xi} & 
(p\bullet i)\circ(p'\bullet i) } }
\end{sideways}\quad\quad 
\begin{sideways}
\scalebox{0.75}{
\xymatrix{
p\bullet p' \bullet p'' \bullet i \ar[r]^-{\varrho\bullet 1 \bullet 1\bullet 1}
\ar[d]_-{1\bullet \varrho'\bullet 1\bullet 1} &
(p\circ j)\bullet p' \bullet p'' \bullet i\ar[rr]^-\xi
\ar[d]^-{1\bullet \varrho'\bullet 1\bullet 1} &&
(p\bullet i) \circ (p' \bullet p'' \bullet i ) 
\ar[d]^-{1\circ (\varrho'\bullet 1\bullet 1)}\\
p\bullet (p'\circ j)\bullet p'' \bullet i 
\ar[r]^-{\varrho\bullet 1 \bullet 1\bullet 1}
\ar[d]_-{\varrho\bullet 1 \bullet 1\bullet 1} &
(p\circ j)\bullet (p'\circ j)\bullet p'' \bullet i \ar[rr]^-\xi
\ar[rd]^-{1\bullet \xi} \ar[d]^-{(1\circ \xi^0) \bullet 1 \bullet 1\bullet 1}
&&
(p\bullet i)\circ ((p'\circ j)\bullet p'' \bullet i )
\ar[d]^-{1 \circ \xi}\\
(p\circ j)\bullet (p'\circ j)\bullet p'' \bullet i 
\ar[r]^-{(\varrho\circ1)\bullet 1 \bullet 1\bullet 1}
\ar[d]_-{\xi \bullet 1 \bullet 1} &
(p\circ j\circ j)\bullet (p'\circ j)\bullet (p'' \bullet i)
\ar[rd]^-{1\bullet \xi} \ar[d]^-{\xi\bullet 1} &
(p\circ j)\bullet((p'\bullet i) \circ (p''\bullet i)) \ar[r]^-\xi
\ar[d]^-{(1\circ \xi^0) \bullet 1} &
(p \bullet i )\circ (p'\bullet i) \circ (p''\bullet i)
\ar[d]^-{1\circ(\xi^0 \bullet 1)}\ar@{=}@/^8pc/[dd]\\
((p\bullet p')\circ j)\bullet p'' \bullet i \ar[d]_-\xi &
(((p\circ j)\bullet p')\circ j)\bullet(p''\bullet i) \ar[d]^-\xi & 
(p\circ j\circ j)\bullet ((p' \bullet i)\circ (p''\bullet i)) \ar[r]^-\xi 
\ar[d]^-\xi &
(p\bullet i) \circ ((j\circ j)\bullet ((p' \bullet i)\circ (p''\bullet i))) 
\ar[d]^-{1\circ\xi}\\
(p\bullet p' \bullet i)\circ (p''\bullet i)
\ar[r]_-{(\varrho\bullet 1 \bullet 1)\circ1} &
((p\circ j)\bullet p'\bullet i)\circ(p''\bullet i) \ar@{=}[r] & 
((p\circ j)\bullet p'\bullet i)\circ(p''\bullet i) \ar[r]_-{\xi\circ 1} &
(p \bullet i)\circ (p'\bullet i)\circ (p''\bullet i)}}
\end{sideways}
\caption{Diagrams for the proof of the monoidality of the equivalence
  $\cm(M,M)^j\simeq \cm(M,M)_i\rev$ }
\label{fig:monoidal_equivalence}
\end{figure}
\end{proof}

\subsection{Fundamental theorem of Hopf modules}

In \cite{BohmChenZhang}, a formulation of the fundamental theorem of Hopf 
modules was given in the context of duoidal categories. If \cd is a duoidal
category, and $a$ is a bimonoid in \cd, then the category $\cd^a$ of
(say, right) $a$-comodules is monoidal with respect to $\bullet$,
and the cofree comodule $(a,\delta)$ is a monoid in the monoidal category
$\cd^a$; thus we may consider the category $\cd^a_a$ of  (say, right)
$a$-modules in $\cd^a$. There is a canonical comparison 
functor $K\colon \cd^j\to \cd^a_a$ sending a $j$-comodule $(p,\rho\colon p\to
p\circ j)$ to the free $a$-module $p\bullet a$, equipped with $a$-comodule
structure 
$$ 
\xymatrix{
p\bullet a \ar[r]^-{\rho\bullet\delta} & 
(p\circ j)\bullet(a\circ a) \ar[r]^-{\xi} & 
(p\bullet a)\circ(j\bullet a) \ar@{=}[r] & (p\bullet a)\circ a }
$$
with the obvious action on morphisms.  {\em The fundamental theorem of Hopf
modules for $a$} is the assertion that this functor $K$ is an equivalence of
categories.

Theorem~3.11 of \cite{BohmChenZhang} includes the assertion that if 
idempotent morphisms in $\mathcal D$ split and  the functor $FV\colon
\cd^j\to\cd_i$ of \eqref{eq:i-j_adjunction}  is fully faithful, then the
fundamental theorem for $a$ holds if and only if the Galois map
$\beta_{q,j}\colon(q\circ j)\bullet a\to  q\circ a$ in \eqref{eq:beta} is
invertible for every $a$-module $q$.

Since we saw in Proposition~\ref{prop:trivial-cotrivial} that for a duoidal
category $\cm(M,M)$ arising from a naturally Frobenius map-monoidale
$(M,m,u)$, the adjunction $FV\dashv GU$ is in fact an equivalence, we may
apply \cite[Theorem~3.11]{BohmChenZhang} to deduce: 

\begin{theorem}\label{thm:fthm}
If $(M,m,u)$ is a naturally Frobenius map-monoidale in a monoidal bicategory
\cm  in which idempotent 2-cells split, and $a$ is a
monoidal comonad on $(M,m,u)$, then the following conditions are equivalent: 
\begin{enumerate}[(a)]
\item the Galois maps $\beta_{q,j}\colon (q\circ j)\bullet a\to q\circ a$ of
  \eqref{eq:beta} are invertible for every $a$-module $q$;  
\item the fundamental theorem of Hopf modules for $a$ holds, in the sense that
  the functor $K\colon\cm(M,M)^j\to\cm(M,M)^a_a$ is an equivalence. 
\end{enumerate}
\end{theorem}

Since the authors of \cite{BohmChenZhang} work with general duoidal
categories, they have an extra duality which is not available to us: the
opposite $\cd\op$ of a duoidal category \cd is also duoidal, with the roles of
$\circ$ and $\bullet$ interchanged, but the notion of bimonoid unchanged. The
{\em dual fundamental theorem of Hopf modules for $a$} is the assertion that
the comparison  functor $K'\colon \cd_i\to\cd^a_a$, sending an
$i$-module $(q,\gamma)$ to the cofree $a$-comodule $q\circ a$, equipped with
$a$-module structure
$$
\xymatrix{
(q\circ a)\bullet a \ar@{=}[r] & 
(q\circ a)\bullet(i\circ a) \ar[r]^-{\xi} & 
(q\bullet i)\circ(a \bullet a) \ar[r]^-{\gamma\circ\mu} & 
q\circ a, }
$$
is an equivalence of categories. The dual of
\cite[Theorem~3.11]{BohmChenZhang}, formulated explicitly as their 
Theorem~3.14, states that if  idempotent morphisms in $\mathcal D$ split
and $GU\colon\cd_i\to\cd^j$ is fully faithful, then
the dual fundamental theorem for $a$ holds if and only if the co-Galois map
$\zeta_{p,i}\colon p\bullet a\to (p\bullet i)\circ a$ of \eqref{eq:zeta} is
invertible for every comodule $p$.  

Once again, in our context $GU$ is an equivalence by
Proposition~\ref{prop:trivial-cotrivial}, and so we deduce: 

\begin{theorem}\label{thm:dualfthm}
If $(M,m,u)$ is a naturally Frobenius map-monoidale in a monoidal bicategory
\cm  in which idempotent 2-cells split, and $a$ is a
monoidal comonad on $(M,m,u)$, then the following conditions are equivalent: 
\begin{enumerate}[(a)]
\item the co-Galois maps $\zeta_{p,i}\colon p\bullet a\to
(p\bullet i)\circ a$ of \eqref{eq:zeta} are invertible for every $a$-comodule
  $p$;
\item the dual fundamental theorem of Hopf modules for $a$ holds, in the sense
  that the functor $K'\colon\cm(M,M)_i\to\cm(M,M)^a_a$ is an equivalence. 
\end{enumerate}
\end{theorem}

\subsection{Summary}

In this brief section we combine all the various results about Hopf-like
conditions into a single statement. 

\begin{theorem}\label{thm:summary}
Let $(M,m,u)$ be a naturally Frobenius map-monoidale in monoidal
bicategory \cm in which idempotent 2-cells split, and let $a$ be a monoidal
comonad on $(M,m,u)$. If the object $M$ is well-pointed (equivalently,
well-copointed) then the following conditions are equivalent:   
\begin{enumerate}[(a)]
\item $a$ has an antipode $\ant\colon a\to a^-$ in the sense of
  Theorem~\ref{thm:antipode}; 
\item the Hopf map $\hat{\beta}$ of \eqref{eq:beta-hat} is invertible; 
\item the co-Hopf map $\hat{\zeta}$ of \eqref{eq:zeta-hat} is invertible; 
\item the Galois maps $\beta_{q,x}\colon (q\circ x)\bullet a\to q\circ (x\bullet
  a)$ of \eqref{eq:beta} are invertible for every 1-cell $x$ and every
  module $q$;
\item the Galois maps $\beta_{q,j}\colon (q\circ j)\bullet a\to q\circ a$ are
  invertible for every module $q$; 
\item the co-Galois maps $\zeta_{p,x}\colon p\bullet(x\circ a)\to
  (p\bullet x)\circ a$ of \eqref{eq:zeta} are invertible for every 1-cell $x$
  and every comodule $p$;
\item the co-Galois maps $\zeta_{p,i}\colon p\bullet a\to
  (p\bullet i)\circ a$ are invertible for every comodule $p$;
\item the fundamental theorem for Hopf modules holds for $a$, in the sense of
  Theorem~\ref{thm:fthm}; 
\item the dual fundamental theorem for Hopf modules holds for $a$, in the
  sense of Theorem~\ref{thm:dualfthm}. 
\end{enumerate}
\end{theorem}

\section{Back to the examples}

The aim of this final section is to draw conclusions from Theorem
\ref{thm:summary} in the examples of Section \ref{sec:examples}.

\subsection{Hopf algebras in braided monoidal categories}

Applying Theorem \ref{thm:summary} to a monoidal comonad in Example
\ref{ex:braided_bialg}; that is, to a bialgebra in a braided monoidal
category, we obtain a variant of Theorem 3.6 in \cite{Vercruysse}. 

\subsection{Groupoids}
Let us apply Theorem \ref{thm:summary} to the monoidal comonad in Example
\ref{ex:category}; that is, to a small category $a$. Then $a^-$ is the
opposite category $a^{\op}$ and the `antipode' in part (a) of Theorem
\ref{thm:summary} is the same as the inverse operation $a\to a^{\op}$. That is
to say, part (a) of Theorem \ref{thm:summary} asserts that $a$ is a {\em
groupoid}. Thus Theorem \ref{thm:summary} provides a generalization, and an
alternative proof, of Corollary 4.6 in 
\cite{BohmChenZhang}. 

\subsection{Hopf algebroids}
Next we apply Theorem \ref{thm:summary} to the monoidal comonad in
Example \ref{ex:bialgebroid}; that is, to a bialgebroid $a$ over a commutative
algebra $R$ (such that the source and target maps land in the center of
$a$). Then $a^-$ is the opposite $R$-bimodule (whose actions are
obtained interchanging  the left and right actions on $a$), and the
`antipode' in part (a) of Theorem \ref{thm:summary} is the same as the
antipode in the sense of \cite{Ravenel} (see also \cite{Bohm:HoA} and the
references therein). That is to say, part (a) of Theorem \ref{thm:summary}
asserts that $a$ is a {\em Hopf algebroid}. Thus Theorem \ref{thm:summary}
provides a generalization, and an alternative proof, of Corollary 4.10 in
\cite{BohmChenZhang}. 

\subsection{Weak Hopf algebras} 
Next, we apply Theorem \ref{thm:summary} to the monoidal comonad in 
Example \ref{ex:weak-bialgebra}; that is, to a weak bialgebra $a$ with
separable Frobenius base algebra $R$. Then the $R^{\mathrm{op}} \ox
R$-bimodule $a^-$ lives on the same vector space $a$ but the actions are
twisted (with the help of the Nakayama automorphism of $R$). The `antipode' in
part (a) of Theorem \ref{thm:summary} is the same as the antipode in the sense
of \cite{BNSz:WHAI}. However, the Hopf modules appearing in parts (h) and (i)
of Theorem  \ref{thm:summary} are different from the Hopf modules discussed in
\cite{BNSz:WHAI}. Hence the characterizations of weak Hopf algebras given  
in Theorem~\ref{thm:summary} (h) and (i) are not literally the same as
those in \cite{BNSz:WHAI}, although each can be deduced from the other.  

\subsection{Brugui\`eres-Virelizier antipode}

We are grateful to Ignacio L\'opez-Franco for the suggestion that we compare 
our antipodes to those introduced by Brugui\`eres and Virelizier in 
\cite[Section~3.3]{BV}. This section is the result of that suggestion. 

Let $(M,\ox,i)$ be a monoidal category with left and right
duals. For an object $x$, we write $\ld{x}$ for the left dual of $x$, which
includes morphisms $e_x\colon\ld{x}\ox x\to i$ and $d_x\colon i\to x\ox
\ld{x}$ satisfying the triangle equations. 

For a functor $f\colon M\to M$, seen as a profunctor, the mate
$f^-$ of Section~\ref{sec:duality&map-monoidales}  is given by  
\begin{align*}
f^-(y,x) &\cong \int^{z \in M} 
M(z \otimes y ,i) \x M(i, x \otimes f(z )  )  \\ 
&\cong \int^{z \in M} 
M(z,\ld{y})\x M(\ld{x},f(z)) 
\cong M(\ld{x},f(\ld{y}))
\end{align*}
and so to give a 2-cell $f\to f^-$  in \Prof  is to
give maps $M(y,f (  x ))\to M(\ld{x},f(\ld{y}))$,
natural in $x$ and $y$; or equivalently to give morphisms 
$\ld{x}\to f(\ld{f(x)})$ natural in $x$.

Now suppose that $a$ is a monoidal comonad on $M$, seen as a bimonoid in
$\Prof(M,M)$, and that $s\colon\ld{x}\to a(\ld{a(x)})$ determines a 
2-cell  $\ant\colon a\to a^-$  in \Prof. The two axioms for
$\ant$ to be an antipode say that the diagrams 
\begin{equation}\label{eq:BV_antipode}
\xymatrix{
 i \ar[r]^-{d_{a(x)}} \ar[d]_-{a_0} & a(x)\ox\ld{a(x)} \ar[r]^-{1\ox s_{a(x)}} & 
 a(x)\ox a(\ld{a^2(x)}) \ar[r]^-{1\ox a(\ld{{\delta_{x}}})} & 
 a(x)\ox a(\ld{a(x)}) \ar[d]^-{a_2} \\
 a(i)   \ar[r]_-{a(d_x)} & a(x\ox\ld{x}) 
 \ar[rr]_-{a(1\ox\ld{{\epsilon_{x}}})} && a(x\ox\ld{a(x)}) 
\\
a(x) \ar[r]^-{d_x\ox 1} \ar[d]_-{\epsilon_{x}} & 
x \ox  \ld{x}  \ox a(x) 
\ar[r]^-{1  \ox s_{x} \ox \delta_{ x}} &  
x \ox  a(\ld{a(x)}) \ox a^2(x) 
\ar[r]^-{1 \ox a_2} & 
x \ox a(\ld{a(x)} \ox a(x)) 
\ar[d]^-{1 \ox a(e_{a(x)})} \\
x \ar[r]_-{ \cong} & x \ox i \ar[rr]_-{1 \ox a_0} &&  
x \ox  a(i) 
}
\end{equation}
commute. The second of these is equivalent, via the duality $\ld{x}\dashv x$,
to commutativity of the diagram  
$$\xymatrix{
\ld{x} \ox a(x) 
\ar[r]^-{s_{ x} \ox \delta_{ x}} 
\ar[d]_-{1 \ox \epsilon_{ x}} & 
a(\ld{a(x)}) \ox a^2(x) \ar[r]^-{a_2} & 
a(\ld{a(x)} \ox a(x)) \ar[d]^-{a(e_{a(x)})} \\
\ld{x} \ox x \ar[r]_-{e_x}  & i \ar[r]_-{a_0} & a(i). }$$

Now a monoidal comonad on $M$ can be seen as an opmonoidal monad on $M\op$ or
$M\oprev$, and the first axiom in \eqref{eq:BV_antipode} together with
the reformulation of the second axiom given above shows that $a$ has
an antipode as in part (a) of Theorem \ref{thm:summary} if and only if the 
corresponding opmonoidal monad on $M\oprev$ has a left antipode in the sense 
of Brugui\`eres and Virelizier. Thus our Theorem~\ref{thm:antipode} 
is a generalization of \cite[Theorem~3.10]{BLV:HopfMonad}.

Brugui\`eres and Virelizier also prove a fundamental theorem of Hopf modules
\cite[Theorem~4.6]{BV}, but this seems to be a different theorem. 
\appendix
\section{Duoidal structure of the duality functor} 
\label{app:Xi-Upsilon}

For the convenience of the reader, in this appendix we record the duoidal
structure of the functor $(-)^-:\cm(M,M) \to \cm\oprev (M,M)$ for a naturally
Frobenius map-monoidale $M$ in a monoidal bicategory \cm; see
Section~\ref{sec:duality&map-monoidales}.

In the case of the monoidal structure involving $\circ$, we write
$\Xi=\Xi_{g,f}\colon f^-\circ g^-\cong (g \circ f)^-$ and $\Xi_0\colon i\cong
i^-$ for the structure maps. Explicitly, they are given by the pasting
composites  
$$\xymatrix{
M \ar[r]^{1u} \ar[d]_{1u} & 
M^2 \ar[r]^{1 m^*} & 
M^3 \ar[r]^{1f1} \ar[d]_{111u}& 
M^3 \ar[r]^{m1} & 
M^2 \ar[r]^{u^*1} & 
M \ar[d]^{1u}  \\
M^2 \ar[d]_{1m^*} && 
M^4 \ar[d]_{111 m^*} &&& 
M^2 \ar[d]^{1 m^*} \\
M^3 \ar[d]_{1g1} && 
M^5 \ar[d]_{111g1} &&& 
M^3 \ar[d]^{1g1} \\
M^3 \ar[r]^{1u11} \ar@{=}[dr] & 
M^4 \ar[d]^{1m1} \ar[r]^{1m^*11}_{~}="1" & 
M^5 \ar[r]^{1f111} \ar[d]^{11m1} & 
M^5 \ar[r]^{m111} & 
M^4 \ar[r]^{u^* 111} & 
M^3 \ar[d]^{m1} \\
& M^3 \ar[r]_{1m^*1}^{~}="2" \ar@{=}[dr] & M^4 
\ar[d]^{11u^*1} &  
& 
& M^2 \ar[d]^{u^* 1} \\
&& M^3 \ar[r]_{1f1} & 
M^3 \ar[r]_{m1} & 
M^2 \ar[r]_{u^*1} & M 
\ar@{=>}"1";"2"^{\pi'}  
}
\xymatrix{
M \ar[r]^{1u} \ar@{=}[dr] & M^2 \ar[r]^{1m^*}_{~}="1" \ar[d]^{m} & 
M^3 \ar[d]^{m1} \\
& M \ar[r]_{m^*}^{~}="2" \ar@{=}[dr] & M^2 \ar[d]^{u^*1} \\
&& M 
\ar@{=>}"2";"1"_{\pi^{-1}}
}$$
in which all the larger regions contain pseudonaturality isomorphisms, the
triangles contain unit/counit isomorphisms, and the squares a Frobenius
isomorphism $\pi'$ or $\pi^{-1}$. 

For the $\bullet$-product there is an isomorphism
$\Upsilon=\Upsilon_{f,g}\colon g^-\bullet f^-\cong(f\bullet g)^-$ which can be
constructed from the isomorphisms  
$$
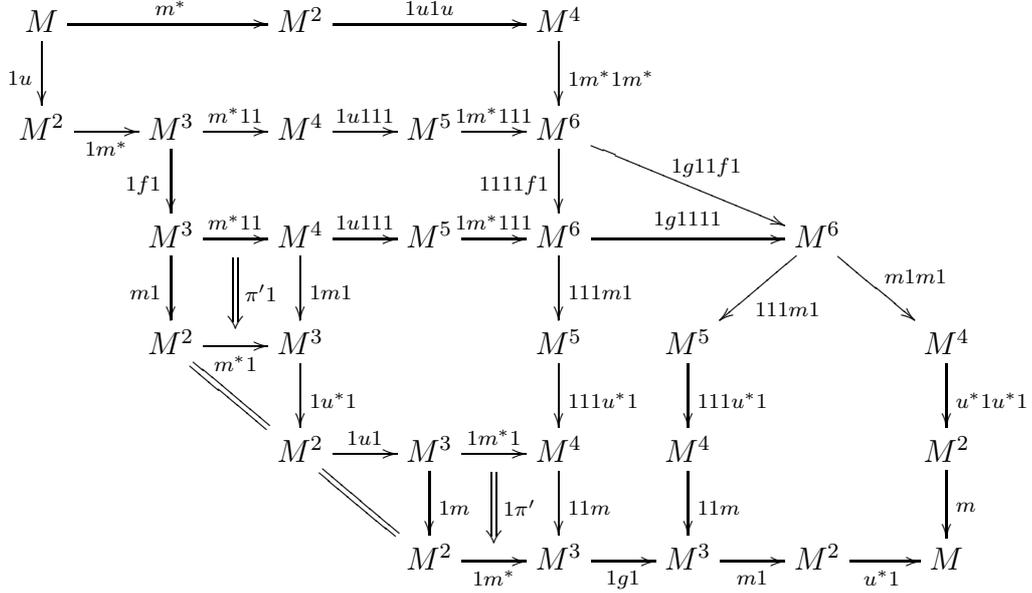
\xymatrix{
M \ar[rr]^{m^*} \ar[d]_{1u} && 
M^2 \ar[rr]^{1u1u} && 
M^4 \ar[d]^{1m^*1m^*} \\
M^2 \ar[r]_{1m^*} &
M^3 \ar[r]^{m^*11} \ar[d]_{1f1} & 
M^4 \ar[r]^{1u111} & 
M^5 \ar[r]^{1m^*111} &
M^6  \ar[d]_{1111f1} \ar[drr]^{1g11f1} \\
&
M^3 \ar[r]^{m^*11}_{~}="1" \ar[d]_{m1} &
M^4 \ar[r]^{1u111} \ar[d]^{1m1} & 
M^5 \ar[r]^{1m^*111} & 
M^6 \ar[rr]^{1g1111} \ar[d]^{111m1} &&
M^6 \ar[dr]^{m1m1} \ar[dl]^{111m1} \\
&
M^2 \ar[r]^{~}="2"_{m^*1} \ar@{=}[dr] & 
M^3 \ar[d]^{1u^*1} && 
M^5 \ar[d]^{111u^*1}  & 
M^5 \ar[d]^{111u^*1} && 
M^4 \ar[d]^{u^*1u^*1} \\
&& M^2 \ar[r]^{1u1}\ar@{=}[dr]& 
M^3 \ar[r]^{1m^*1}_{~}="3" \ar[d]^{1m}  & 
M^4 \ar[d]^{11m} & 
M^4 \ar[d]^{11m} && 
M^2 \ar[d]^{m} \\
&&& M^2 \ar[r]_{1m^*}^{~}="4" & 
M^3 \ar[r]_{1g1} & 
M^3 \ar[r]_{m1}& 
M^2 \ar[r]_{u^*1} & M 
\ar@{=>}"1";"2"^{\pi'1}
\ar@{=>}"3";"4"^{1\pi'} }
$$
and 
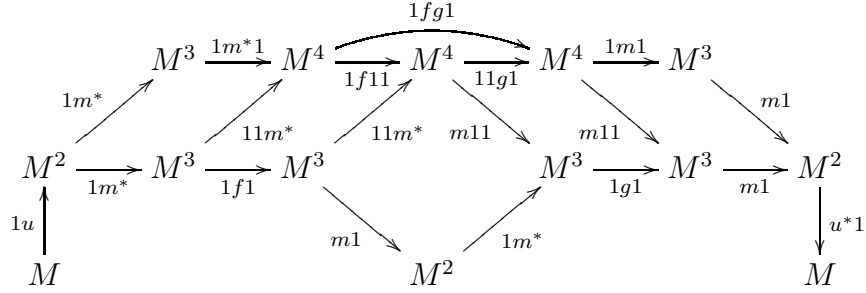
$$
\xymatrix{
& M^3 \ar[r]^{1m^*1} & M^4\ar[r]_{1f11} \ar@/^1pc/[rr]^{1fg1} & 
M^4 \ar[r]_{11g1} \ar[dr]_{m11} & 
M^4 \ar[r]^{1m1} \ar[dr]_{m11}  & 
M^3 \ar[dr]^{m1} \\
M^2 \ar[ur]^{1m^*} \ar[r]_{1m^*} & 
M^3 \ar[r]_{1f1} \ar[ur]_{11m^*} & 
M^3 \ar[ur]_{11m^*} \ar[dr]_{m1} && 
M^3 \ar[r]_{1g1} & M^3 \ar[r]_{m1} & 
M^2 \ar[d]^{u^*1} \\
M \ar[u]^{1u} &&& 
M^2 \ar[ur]_{1m^*} &&& 
M }
$$
while the isomorphism $\Upsilon_0\colon j^{-}\cong j$ may be constructed as in
the diagram 
$$
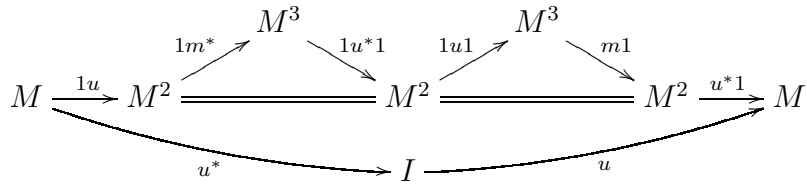
\xymatrix @R1pc {
&& M^3 \ar[dr]^{1u^*1} && 
M^3 \ar[dr]^{m1} \\
M \ar[r]^{1u} \ar@/_0.5pc/[drrr]_{u^*} & 
M^2 \ar[ur]^{1m^*} \ar@{=}[rr] && 
M^2 \ar[ur]^{1u1} \ar@{=}[rr] && 
M^2 \ar[r]^{u^*1} & M \\
&&& I \ar@/_0.5pc/[urrr]_{u} }
$$
using unit and counit isomorphisms for $M$ and pseudofunctoriality of the
tensor product in $\cm$.

\bibliographystyle{plain}

\end{document}